\documentclass[3p,preprint,review,12pt]{elsarticle}
\usepackage{amssymb}
\usepackage{appendix}
\usepackage{verbatim}
\usepackage{amsthm}
\usepackage{lineno,hyperref}
\modulolinenumbers[5]

\usepackage{amsfonts,amsmath,latexsym}
\usepackage[active]{srcltx}
\usepackage{enumitem}

\newtheorem{theorem}{Theorem}[section]
\newtheorem{problem}[theorem]{Problem}
\newtheorem{remark}[theorem]{Remark}

\newproof{pf}{Proof}

\newcommand{\RR}{{\if mm {\rm I}\mkern -3mu{\rm R}\else \leavevmode
		\hbox{I}\kern -.17em\hbox{R} \fi}}

\newcommand{\bu}{\mbox{\boldmath{$u$}}}
\newcommand{\bt}{\mbox{\boldmath{$t$}}}
\newcommand{\bcero}{\mbox{\boldmath{$0$}}}

\newcommand{\bx}{\mbox{\boldmath{$\bx$}}}
\newcommand{\by}{\mbox{\boldmath{$\by$}}}
\newcommand{\bv}{\mbox{\boldmath{$v$}}}

\newcommand{\var}{\varepsilon}

\newcommand{\ba}{\mbox{\boldmath{$a$}}}

\newcommand{\fb}{{{f}}}
\newcommand{\bg}{\mbox{\boldmath{$g$}}}
\renewcommand{\by}{\mbox{\boldmath{$y$}}}
\renewcommand{\bx}{\mbox{\boldmath{$x$}}}
\newcommand{\be}{\mbox{\boldmath{$e$}}}
\newcommand{\bn}{\mbox{\boldmath{$n$}}}
\newcommand{\bh}{{{h}}}
\newcommand{\bbf}{\mbox{\boldmath{$f$}}}
\newcommand{\bbh}{\mbox{\boldmath{$h$}}}

\newcommand{\btheta}{\mbox{\boldmath{$\theta$}}}

\newcommand{\beeta}{\mbox{\boldmath{$\eta$}}}
\newcommand{\bxi}{\mbox{\boldmath{$\xi$}}}

\newcommand{\bTheta}{\mbox{\boldmath{$\Theta$}}}
\newcommand{\bUcal}{\mbox{\boldmath{$\hat{u}$}}}

\newcommand{\Gamae}{\Gamma^{\varepsilon }}

\renewcommand{\d}{\partial}
\newcommand{\eij}{e_{i||j}}
\newcommand{\deij}{\dot{e}_{i||j}}
\newcommand{\ekl}{e_{k||l}}
\newcommand{\dekl}{\dot{e}_{k||l}}
\newcommand{\eab}{e_{\alpha||\beta}}
\newcommand{\est}{e_{\sigma||\tau}}
\newcommand{\dest}{\dot{e}_{\sigma||\tau}}
\newcommand{\estres}{e_{\sigma||3}}

\newcommand{\destres}{\dot{e}_{\sigma||3}}
\newcommand{\eatres}{e_{\alpha||3}}
\newcommand{\edtres}{e_{3||3}}
\newcommand{\deab}{\dot{e}_{\alpha||\beta}}
\newcommand{\deatres}{\dot{e}_{\alpha||3}}
\newcommand{\dedtres}{\dot{e}_{3||3}}
\newcommand{\gab}{\gamma_{\alpha\beta}}
\newcommand{\gst}{\gamma_{\sigma\tau}}
\newcommand{\rab}{\rho_{\alpha\beta}}

\renewcommand{\a}{a^{\alpha\beta\sigma\tau}}
\renewcommand{\b}{b^{\alpha\beta\sigma\tau}}
\renewcommand{\c}{c^{\alpha\beta\sigma\tau}}
\newcommand{\aeps}{a^{\alpha\beta\sigma\tau,\varepsilon}}
\newcommand{\beps}{b^{\alpha\beta\sigma\tau,\varepsilon}}
\newcommand{\ceps}{c^{\alpha\beta\sigma\tau,\varepsilon}}
\renewcommand{\ae}{\ a.e. \ t\in(0,T)}
\newcommand{\aes}{\ a.e. \ \textrm{in} \ (0,T)}
\newcommand{\forallt}{\ \forall  \ t\in[0,T]}
\newcommand{\Forallt}{\ \textrm{for all} \ t\in[0,T]}

\newcommand{\deb}{\rightharpoonup}
\newcommand{\en}{ \ \textrm{in} \ }
\newcommand{\on}{ \ \textrm{on} \ }
\newcommand{\into}{\int_{\omega}}
\newcommand{\intO}{\int_{\Omega}}

\newcommand{\ten}{(a^{\alpha \sigma}a^{\beta \tau} + a^{\alpha \tau}a^{\beta\sigma})}

\newcommand{\WLO}{H^{1}(0,T;L^2(\Omega))}
\newcommand{\WLo}{H^{1}(0,T;L^2(\omega))}
\newcommand{\WHO}{H^{1}(0,T;H^1(\Omega))}
\newcommand{\WHOt}{H^{1}(0,T;[H^1(\Omega)]^3)}
\newcommand{\WHo}{H^{1}(0,T;H^1(\omega))}
\newcommand{\WVo}{H^{1}(0,T;V_M(\omega))}
\newcommand{\WVO}{H^{1}(0,T;V(\Omega))}
\newcommand{\LLO}{L^2(0,T;L^2(\Omega))}

\newcommand{\WHMO}{H^{1}(0,T;H^{-1}(\Omega))}
\newcommand{\WHMOt}{H^{1}(0,T;[H^{-1}(\Omega)]^3)}
\newcommand{\WVMO}{H^{1}(0,T;V^{\#}_M(\Omega))}
\newcommand{\WVMo}{H^{1}(0,T;V^{\#}_M(\omega))}

\newcommand{\mest}{\overline{e_{\sigma||\tau}}}

\newcommand{\mests}{\overline{e_{\sigma||\tau}(s)}}
\newcommand{\dmest}{\dot{\overline{e_{\sigma||\tau}}}}

\newcommand{\meab}{\overline{e_{\alpha||\beta}}}

\newcommand{\dmeab}{\dot{\overline{e_{\alpha||\beta}}}}

\newcommand{\eitres}{e_{i||3}}
\newcommand{\bw}{\mbox{\boldmath{$w$}}}
\newcommand{\bG}{\mbox{\boldmath{$G$}}}

\journal{Journal}
\begin{document}

\begin{frontmatter}

\title{Mathematical justification of a viscoelastic generalized membrane problem}

\author[compostela]{G.~Casti\~neira\corref{mycorrespondingauthor}}
\cortext[mycorrespondingauthor]{Corresponding author}
\ead{gonzalo.castineira@uca.es}
\address[compostela]{ Escuela Polit\'ecnica Superior de Algeciras, Avda. Ramón Puyol s/n, 11202,   Departamento de Matem\'aticas,
	Univ. de C\'adiz, Spain}

\author[corunha]{\'A. Rodr\'{\i}guez-Ar\'os}
\ead{angel.aros@udc.es}
\address[corunha]{E.T.S. N\'autica e M\'aquinas
Paseo de Ronda, 51, 15011, Departamento de  Matem\'aticas, Univ. da Coru\~na, Spain}

\begin{abstract}
We consider a family of linearly viscoelastic shells with thickness $2\var$, clamped along a portion of their lateral face, all having the same middle surface $S=\btheta(\bar{\omega})\subset\RR^3$, where $\omega\subset\RR^2$ is a bounded and connected open set with a Lipschitz-continuous boundary $\gamma$. 
We show that, if the applied body force density is $\mathcal{O}(1)$ with respect to $\var$ and surface tractions density is $\mathcal{O}(\var)$, the solution of the scaled variational problem in curvilinear coordinates, defined over the fixed domain $\Omega=\omega\times(-1,1)$, converges in {\it ad hoc}  functional spaces to a limit $\bu$ as $\var\to 0$ . Furthermore, the average $\overline{\bu(\var)}= \frac1{2}\int_{-1}^{1}\bu (\var) dx_3$, converges in an \textit{ad hoc} space to the unique solution of  what we have identified as (scaled) two-dimensional equations of a viscoelastic generalized membrane  shell, which includes a long-term memory that takes into account previous deformations. We finally provide convergence results which justify those equations.
\end{abstract}


\end{frontmatter}
\linenumbers

\section{Introduction}\setcounter{equation}{0}

In solid mechanics, the obtention of  models for rods, beams,
plates and shells is based on {\it a priori} hypotheses on the
displacement and/or stress fields which, upon substitution in the three-dimensional
equilibrium and constitutive equations, lead to useful simplifications. Nevertheless, from both
constitutive and geometrical point of views, there is a need to
justify the validity of most of the models obtained in this way.

For this reason a considerable effort has been made in the past decades by many authors in order
to derive new models and justify the existing ones by
using the asymptotic expansion method, whose foundations can be
found in \cite{Lions}. Indeed, the first applied results were obtained with the justification of the linearized theory of plate bending in \cite{CD,Destuynder}.

A complete theory regarding  elastic shells can be found in \cite{Ciarlet4b}, where models  for elliptic membranes (see also \cite{CiarletLods,CiarletLods2}), generalized membranes (see \cite{CiarletLods5}) and flexural shells (see \cite{CiarletLods4}) are presented. It contains a full description of the asymptotic procedure that leads to the corresponding sets of two-dimensional equations. Also, the dynamic case has been studied in \cite{Limin_mem,Limin_flex,Limin_koit}, concerning the justification of dynamic equations for membrane, flexural and Koiter shells.  Furthermore, the limit of the three-dimensional unilateral, frictionless, in \cite{ArosObs,ARA2018,AA_contact_shell} we find a contact problem study for elastic elliptic shells where a two-dimensional obstacle problem is derived using asymptotic methods. Even more recently,  we  find the obtention of error estimates for the membrane case in \cite{Cao_Aros}, a  convergence study for elastic elliptic membrane shells in normal compliance contact with a deformable obstacle in \cite{ArosCao19},  and an asymptotic analysis of thermoelastic shells in normal damped response contact in \cite{TGAS_2020}.

A large number of real problems had made it necessary the study of new models which could take into account effects such as hardening and memory of the material. An example of these are the  viscoelasticity models (see \cite{DL,LC1990,Pipkin}). Many authors have contributed to the nowadays knowledge of this sort of problems, providing justified models and results.  Indeed, we can find examples in the literature as \cite{contact2,contact3,SofoneaAMS,contact7,contact1,ASV06}  and in the references therein, a variety of models for  problems concerning the viscoelastic behaviour of the material.  In particular, there exist 
studies of the behaviour of viscoelastic plates as in  \cite{jaru1,plate12}, where models for  von K\'arm\'an plates are analysed. In some of these works, we can find  analysis of the influence of short or long term memory in the equations modelling a problem. These terms  take into account previous deformations of the body, hence, they are commonly presented in some viscoelastic problems. For instance, on one hand, we can find in \cite{plate15} models including a short term memory presented by  a system of integro-differential and pseudoparabolic equations describing large deflections on a viscoelastic plate. On the other hand, in \cite{plate13}  a long term memory is considered on the study of the asymptotic behaviour of the solution of a von K\'arm\'an plate when the time variable tends to infinity. Also, in the reference \cite{plate14}, the authors study the effects of great deflections in thin plates covering both short and long term memory cases. Concerning viscoelastic shell problems, in \cite{plate16} we can find different kind of studies where the authors also remark the viscoelastic property of the material of a shell. For the problems dealing with the shell-type equations, there exists a very limited amount of results available, for instance, \cite{shell11} where the authors present a model for a dynamic contact problem where a short memory (Kelvin-Voigt) material is considered. Particularly remarkable is the increasing number of studies of viscoelastic shells problems in order  to reproduce the complex behaviour of tissues in the field of biomedicine. For example, in \cite{Quarteroni} the difficulties of this kind of problems are detailed and even though an one-dimensional model is derived for modelling a vessel wall, the author  comments the possibility of considering two-dimensional models with a shell-type description and  a viscoelastic constitutive law. In this direction, to our knowledge, in \cite{intro2} we gave the first steps towards the justification of existing models of viscoelastic shells and the finding of new ones. By using the asymptotic expansion method, we found a rich variety of cases, depending on the geometry of the middle surface, the boundary conditions and the order of the applied forces. The most remarkable feature was that from the asymptotic analysis of the three-dimensional problems which included a short term memory represented by a time derivative, a long term memory arised in the two-dimensional limit problems, represented by an integral with respect to the time variable. This fact, agreed with previous asymptotic analysis of viscoelastic rods in  \cite{AV,AV2} where an analogous behaviour was presented as well.


In \cite{eliptico,eliptico2}  we  justified  the equations of a viscoelastic membrane shell  where the surface $S$ is elliptic and the boundary condition of place is considered in the whole lateral face of the shell. Therefore, in this paper the main aim is to justify the remaining cases in the group of viscoelastic membrane cases, known as the viscoelastic generalized membrane shell equations. In such a group, we shall distinguish two kinds of membranes, as it will be detailed in following sections.

To be more specific, we prove that the scaled three-dimensional unknown, $\bu(\var)$, converges as the small parameter $\var$ tends to zero in an \textit{ad hoc} functional space and its transversal average converges to  $\bxi^\var$, the unique  solution of the two-dimensional associated problem. Moreover,  unlike the viscoelastic elliptic membrane shells, the limit of the scaled three-dimensional unknown $\bu(\var)$ is not necessary independent of $x_3$, however we find that that $\d_3\bu(\var)\rightarrow \bcero$ in $\WLO$.

We will follow the notation and style of \cite{Ciarlet4b}, where the linear elastic shells are studied.
For this reason, we shall  reference auxiliary results which apply in the same manner to the viscoelastic case. One of the major differences with respect to previous works in elasticity, consists on time dependence, that will lead to ordinary differential equations that we need to solve in order to characterize the zeroth-order approach of the solution.

The structure of the paper is the following: in Section \ref{problema} 
we shall recall the viscoelastic  problem in Cartesian coordinates and  then, considering the problem for a family of viscoelastic shells of thickness $2\var$, we formulate the problem in curvilinear coordinates. In Section \ref{seccion_dominio_ind} we will use a projection map into a reference domain, we will introduce the scaled unknowns and forces and the assumptions on the coefficients. In Section \ref{preliminares} we recall some technical results which will be needed in what follows. In Section \ref{seccion_fuerzas_admisibles} we shall study the completion spaces that will lead to well posed problems for the viscoelastic membrane shell equations. Then, we will introduce an assumption on the applied forces, needed in the convergence analysis.  In Section \ref{seccion_convergencia} we enunciate the two-dimensional equations for a viscoelastic generalized membrane shell and we present the convergence results when the small parameter $\var$ tends to zero, which is the main result of this paper. Then, we present the convergence results in terms of de-scaled unknowns. In Section \ref{conclusiones} we shall present some conclusions, including a comparison between the viscoelastic models and the elastic case studied in \cite{Ciarlet4b} and comment the convergence results for the remaining cases.

\section{The three-dimensional linearly viscoelastic shell problem}\setcounter{equation}{0} \label{problema}

We denote by $\mathbb{S}^d$, where $d=2,3$ in practice, the space of second-order symmetric tensors on $\mathbb{R}^d$, while \textquotedblleft$\ \cdot$ \textquotedblright will represent the inner product and $\left|\cdot\right|$  the usual norm in $\mathbb{S}^d$ and  $\mathbb{R}^d$. In  what follows, unless the contrary is explicitly written, we will use summation convention on repeated indices. Moreover, Latin indices $i,j,k,l,...$, take their values in the set $\{1,2,3\}$, whereas Greek indices $\alpha,\beta,\sigma,\tau,...$, do it in the set  $\{1,2\}$. Also, we use standard notation for the Lebesgue and Sobolev spaces. For a time dependent function $u$, we denote $\dot{u}$ the first derivative of $u$ with respect to the time variable. Recall that  $"\rightarrow"$ denotes strong convergence, while  $"\rightharpoonup" $ denotes weak convergence.


Let ${\Omega}^*$ be a domain of $\mathbb{R}^3$, with a Lipschitz-continuous boundary ${\Gamma^*}=\d{\Omega^*}$. Let ${\bx^*}=({x}_i^*)$ be a generic point of  its closure $\bar{\Omega}^*$ and let ${\d}^*_i$ denote the partial derivative with respect to ${x}_i^*$. Let $dx^*$ denote the volume element in $\Omega^*$,  $d\Gamma^*$ denote the area element along $\Gamma^*$ and  $\bn^*$ denote the unit outer normal vector along $\Gamma^*$. Finally, let $\Gamma^*_0$  and $\Gamma_1^*$ be subsets of $\Gamma^*$ such that $meas(\Gamma_0^*)>0$ and $\Gamma^*_0 \cap \Gamma_1^*=\emptyset.$ 

The set $\Omega^*$ is the region occupied by a deformable body in the absence of applied forces. We assume that this body is made of a Kelvin-Voigt viscoelastic material, which is homogeneous and isotropic,  so that the material is characterized by its Lam\'e coefficients   $\lambda\geq0, \mu>0$ and its viscosity coefficients, $\theta\geq 0,\rho\geq 0$ (see for instance \cite{DL,LC1990,Shillor}).

Let $T>0$ be the time period of observation. Under the effect of applied forces, the  body is deformed and we denote by $u_i^*:[0,T]\times \bar{\Omega}^*\rightarrow \mathbb{R}^3$ the Cartesian components of the displacements field, defined as $\bu^*:=u_i^* \be^{i}:[0,T]\times\bar{\Omega}^* \rightarrow \mathbb{R}^3$, where $\{\be^i\}$ denotes the Euclidean canonical basis in $\mathbb{R}^3$. 
Moreover, we consider that the displacement field vanishes on the set $\Gamma^*_0$. Hence, the  displacements field $\bu^*=(u_i^*):[0,T]\times\Omega^*\longrightarrow \mathbb{R}^3$ is solution of the following three-dimensional problem in Cartesian coordinates.

\begin{problem}\label{problema_mecanico}
	Find $\bu^*=(u_i^*):[0,T]\times\Omega^*\longrightarrow \mathbb{R}^3$ such that,
	\begin{align}\label{equilibrio}
	-\d_j^*\sigma^{ij,*}(\bu^*)&=f^{i,*} \en \Omega^*, \\\label{Dirichlet}
	u_i^*&=0 \on \Gamma^*_0, \\\label{Neumann}
	\sigma^{ij,*}(\bu^*)n_j^*&=h^{i,*} \on \Gamma_1^*,\\ \label{condicion_inicial} 
	\bu^*(0,\cdot)&=\bu_0^* \en \Omega^*,
	\end{align}
	where the functions
	\begin{align*}
	\sigma^{ij,*}(\bu^*):=A^{ijkl,*}e_{kl}^*(\bu^*)+ B^{ijkl,*}e_{kl}^*(\dot{\bu}^*),
	\end{align*}
	are the  components of the linearized stress tensor field and where the functions
	\begin{align*} 
	& A^{ijkl,*}:= \lambda \delta^{ij}\delta^{kl} + \mu\left(\delta^{ik}\delta^{jl} + \delta^{il}\delta^{jk}\right) , 
	\\ 
	& B^{ijkl,*}:= \theta \delta^{ij}\delta^{kl} + \frac{\rho}{2}\left(\delta^{ik}\delta^{jl} + \delta^{il}\delta^{jk}\right) , 
	\end{align*}
	are the  components of the three-dimensional elasticity and viscosity fourth order tensors, respectively, and 
	\begin{align*}
	e^*_{ij}(\bu^*):= \frac1{2}(\d^*_ju^*_{i}+ \d^*_iu^*_{j}),
	\end{align*}
	designate the  components of the linearized strain tensor associated with the displacement field $\bu^*$of the set $\bar{\Omega}^*$.
\end{problem}
We now proceed to describe the equations in Problem \ref{problema_mecanico}. Expression (\ref{equilibrio}) is the equilibrium equation, where $f^{i,*}$ are the  components of the volumic force densities. The equality (\ref{Dirichlet}) is the Dirichlet condition of place, (\ref{Neumann}) is the Neumann condition, where $h^{i,*}$ are the  components of surface force densities and (\ref{condicion_inicial}) is the initial condition, where $\bu_0^*$ denotes the initial  displacements.

Note that, for the sake of briefness, we omit the explicit dependence on the space and time variables when there is no ambiguity. Let us define the space of admissible unknowns,
\begin{align*} 
V(\Omega^*)=\{\bv^*=(v_i^*)\in [H^1(\Omega^*)]^3; \bv^*=\mathbf{\bcero} \ on \ \Gamma_0^*  \}.
\end{align*}
Therefore, assuming enough regularity,  the unknown  $\bu^*=(u_i^*)$ satisfies the following variational problem in Cartesian coordinates:
\begin{problem}\label{problema_cartesian}
	Find $\bu^*=(u_i^*):[0,T]\times {\Omega}^* \rightarrow \mathbb{R}^3$  such that, 
	\begin{align*} 
	\displaystyle  \nonumber
	& \bu^*(t,\cdot)\in V(\Omega^*) \forallt,
	\\ \nonumber 
	&\int_{\Omega^*}A^{ijkl,*}e^*_{kl}(\bu^*)e^*_{ij}(\bv^*) dx^*+ \int_{\Omega^*} B^{ijkl,*}e^*_{kl}(\dot{\bu}^*)e_{ij}^*(\bv^*)   dx^*
	\\ 
	& \quad= \int_{\Omega^*} f^{i,*} v_i^*  dx^* + \int_{\Gamma_1^*} h^{i,*} v_i^*  d\Gamma^* \quad \forall \bv^*\in V(\Omega^*), \aes,
	\\\displaystyle 
	& \bu^*(0,\cdot)= \bu_0^*(\cdot).
	\end{align*}
\end{problem} 

Let us consider that $\Omega^*$ is a viscoelastic shell of thickness $2\var$. Now, we shall express the equations of the Problem \ref{problema_cartesian} in terms of  curvilinear coordinates. Let $\omega$ be a domain of $\mathbb{R}^2$, with a Lipschitz-continuous boundary $\gamma=\d\omega$. Let $\by=(y_\alpha)$ be a generic point of  its closure $\bar{\omega}$ and let $\d_\alpha$ denote the partial derivative with respect to $y_\alpha$. 

Let $\btheta\in\mathcal{C}^2(\bar{\omega};\mathbb{R}^3)$ be an injective mapping such that the two vectors $\ba_\alpha(\by):= \d_\alpha \btheta(\by)$ are linearly independent. These vectors form the covariant basis of the tangent plane to the surface $S:=\btheta(\bar{\omega})$ at the point $\btheta(\by)=\by^*.$ We can consider the two vectors $\ba^\alpha(\by)$ of the same tangent plane defined by the relations $\ba^\alpha(\by)\cdot \ba_\beta(\by)=\delta_\beta^\alpha$, that constitute the contravariant basis. We define the unit vector, 
\begin{align}\label{a_3}
\ba_3(\by)=\ba^3(\by):=\frac{\ba_1(\by)\wedge \ba_2(\by)}{| \ba_1(\by)\wedge \ba_2(\by)|},
\end{align} 
normal vector to $S$ at the point $\btheta(\by)=\by^*$, where $\wedge$ denotes vector product in $\mathbb{R}^3.$ 

We can define the first fundamental form, given as metric tensor, in covariant or contravariant components, respectively, by
\begin{align*}
a_{\alpha\beta}:=\ba_\alpha\cdot \ba_\beta, \qquad a^{\alpha\beta}:=\ba^\alpha\cdot \ba^\beta,
\end{align*}
the second fundamental form, given as curvature tensor, in covariant or mixed components, respectively, by
\begin{align*}
b_{\alpha\beta}:=\ba^3 \cdot \d_\beta \ba_\alpha, \qquad b_{\alpha}^\beta:=a^{\beta\sigma} b_{\sigma\alpha},
\end{align*}
and the Christoffel symbols of the surface $S$ by
\begin{align*}
\Gamma^\sigma_{\alpha\beta}:=\ba^\sigma\cdot \d_\beta \ba_\alpha.
\end{align*}

The area element along $S$ is $\sqrt{a}dy=dy^*$ where 
\begin{align}\label{definicion_a}
a:=\det (a_{\alpha\beta}).
\end{align}

Let $\gamma_0$ be a subset  of  $\gamma$, such that $meas (\gamma_0)>0$. 
For each $\varepsilon>0$, we define the three-dimensional domain $\Omega^\varepsilon:=\omega \times (-\varepsilon, \varepsilon)$ and  its boundary $\Gamae=\d\Omega^\var$. We also define  the following parts of the boundary, 
\begin{align*}
\Gamma^\varepsilon_+:=\omega\times \{\varepsilon\}, \quad \Gamma^\varepsilon_-:= \omega\times \{-\varepsilon\},\quad \Gamma_0^\varepsilon:=\gamma_0\times[-\varepsilon,\varepsilon].
\end{align*}

Let $\bx^\varepsilon=(x_i^\varepsilon)$ be a generic point of $\bar{\Omega}^\varepsilon$ and let $\d_i^\var$ denote the partial derivative with respect to $x_i^\varepsilon$. Note that $x_\alpha^\varepsilon=y_\alpha$ and $\d_\alpha^\varepsilon =\d_\alpha$. Let $\bTheta:\bar{\Omega}^\varepsilon\rightarrow \mathbb{R}^3$ be the mapping defined by
\begin{align} \label{bTheta}
\bTheta(\bx^\varepsilon):=\btheta(\by) + x_3^\varepsilon \ba_3(\by) \ \forall \bx^\varepsilon=(\by,x_3^\varepsilon)=(y_1,y_2,x_3^\varepsilon)\in\bar{\Omega}^\varepsilon.
\end{align}

The next theorem shows that if the injective mapping $\btheta:\bar{\omega}\rightarrow\mathbb{R}^3$ is smooth enough, the mapping $\bTheta:\bar{\Omega}^\var\rightarrow\mathbb{R}^3$ is also injective for $\var>0$ small enough (see Theorem 3.1-1, \cite{Ciarlet4b}).

\begin{theorem}\label{var_0}
	Let $\omega$ be a domain in $\mathbb{R}^2$. Let $\btheta\in\mathcal{C}^2(\bar{\omega};\mathbb{R}^3)$ be an injective mapping such that the two vectors $\ba_\alpha=\d_\alpha\btheta$ are linearly independent at all points of $\bar{\omega}$ and let $\ba_3$  defined  in (\ref{a_3}). Then, there exists $\var_0>0$ such that for all $\var_1$, $0<\var_1\leq\var_0$   the mapping $\bTheta:\bar{\Omega}_1 \rightarrow\mathbb{R}^3$ defined by
	\begin{align*}
	\bTheta(\by,x_3):=\btheta(\by) + x_3 \ba_3(\by) \ \ \forall (\by,x_3)\in\bar{\Omega}_1, \ \textrm{where} \ \Omega_1:=\omega\times(-\var_1,\var_1),
	\end{align*}
	is a $\mathcal{C}^1-$diffeomorphism from $\bar{\Omega}_1$ onto $\bTheta(\bar{\Omega}_1)$ and $\det (\bg_1,\bg_2,\bg_3)>0$ in $\bar{\Omega}_1$, where $\bg_i:=\d_i\bTheta$. 
\end{theorem}

For each $\var$, $0<\var\le\var_0$, the set $\bTheta(\bar{\Omega}^\var)=\bar{\Omega}^*$ is the reference configuration of a viscoelastic shell, with middle surface $S=\btheta(\bar{\omega})$ and thickness $2\varepsilon>0$.
Furthermore for $\varepsilon>0,$ $\bg_i^\varepsilon(\bx^\varepsilon):=\d_i^\varepsilon\bTheta(\bx^\varepsilon)$ are linearly independent and the mapping $\bTheta:\bar{\Omega}^\varepsilon\rightarrow \mathbb{R}^3$ is injective for all $\var$, $0<\var\le\var_0$, as a consequence of injectivity of the mapping $\btheta$. Hence, the three vectors $\bg_i^\varepsilon(\bx^\varepsilon)$ form the covariant basis of the tangent space at the point $\bx^*=\bTheta(\bx^\varepsilon)$ and $\bg^{i,\varepsilon}(\bx^\varepsilon) $ defined by the relations $\bg^{i,\varepsilon}\cdot \bg_j^\varepsilon=\delta_j^i$ form the contravariant basis at the point $\bx^*=\bTheta(\bx^\varepsilon)$. We define the metric tensor, in covariant or contravariant components, respectively, by
\begin{align*}
g_{ij}^\varepsilon:=\bg_i^\varepsilon \cdot \bg_j^\varepsilon,\quad g^{ij,\varepsilon}:=\bg^{i,\varepsilon} \cdot \bg^{j,\varepsilon},
\end{align*}
and Christoffel symbols by
\begin{align} \label{simbolos3D}
\Gamma^{p,\varepsilon}_{ij}:=\bg^{p,\varepsilon}\cdot\d_i^\varepsilon \bg_j^\varepsilon. 
\end{align}

The volume element in the set $\bTheta(\bar{\Omega}^\varepsilon)=\bar{\Omega}^*$ is $\sqrt{g^\varepsilon}dx^\var=dx^*$ and the surface element in $\bTheta(\Gamma^\varepsilon)=\Gamma^*$ is $\sqrt{g^\varepsilon}d\Gamae =d\Gamma^*$  where
\begin{align} \label{g}
g^\varepsilon:=\det (g^\varepsilon_{ij}).
\end{align} 
Therefore, for a field ${\bv}^*$ defined in $\bTheta(\bar{\Omega}^\var)=\bar{\Omega}^*$, we define its covariant curvilinear coordinates $v_i^\var$ by
\begin{equation*}
{\bv}^*({\bx}^*)={v}^*_i({\bx}^*){\be}^i=:v_i^\var(\bx^\var)\bg^i(\bx^\var),\ {\rm with}\ {\bx}^*=\bTheta(\bx^\var).
\end{equation*}


Besides, we denote by $u_i^\varepsilon:[0,T]\times \bar{\Omega}^\varepsilon \rightarrow \mathbb{R}^3$ the covariant components of the displacements field, that is  $\bUcal^\var:=u_i^\varepsilon \bg^{i,\varepsilon}:[0,T]\times\bar{\Omega}^\varepsilon \rightarrow \mathbb{R}^3$ . For simplicity, we define the vector field $\bu^\varepsilon=(u_i^\varepsilon):[0,T]\times {\Omega}^\varepsilon \rightarrow \mathbb{R}^3$ which will be denoted vector of unknowns.

Recall that we assumed that the shell is subjected to a boundary condition of place; in particular that the displacements field vanishes in  $\bTheta(\Gamma_0^\varepsilon)=\Gamma_0^*$.

Accordingly, let us define the space of admissible unknowns,
\begin{align*} 
V(\Omega^\varepsilon)=\{\bv^\varepsilon=(v_i^\varepsilon)\in [H^1(\Omega^\varepsilon)]^3; \bv^\varepsilon=\mathbf{\bcero} \ on \ \Gamma_0^\varepsilon  \}.
\end{align*}

This is a real Hilbert space with the induced inner product of $[H^1(\Omega^\var)]^3$. The corresponding norm  is denoted by $\left\| \cdot\right\| _{1,\Omega^\var}$. 

Therefore, we can find the expression of the Problem \ref{problema_cartesian} in curvilinear coordinates (see \cite{Ciarlet4b} for details). Hence, the ``displacements" field $\bu^\var=(u_i^\var)$ verifies the following variational problem of a three-dimensional viscoelastic shell in curvilinear coordinates:

\begin{problem}\label{problema_eps}
	Find $\bu^\varepsilon=(u_i^\varepsilon):[0,T]\times {\Omega}^\varepsilon \rightarrow \mathbb{R}^3$  such that, 
	\begin{align} 
	\displaystyle  \nonumber
	& \bu^\varepsilon(t,\cdot)\in V(\Omega^\varepsilon) \forallt,
	\\ \nonumber 
	&\int_{\Omega^\varepsilon}A^{ijkl,\varepsilon}e^\varepsilon_{k||l}(\bu^\varepsilon)e^\varepsilon_{i||j}(\bv^\varepsilon)\sqrt{g^\varepsilon} dx^\varepsilon+ \int_{\Omega^\varepsilon} B^{ijkl,\varepsilon}e^\varepsilon_{k||l}(\dot{\bu}^\varepsilon)e_{i||j}^\var(\bv^\varepsilon) \sqrt{g^\varepsilon}  dx^\varepsilon
	\\ \label{Pbvariacionaleps}
	& \quad= \int_{\Omega^\varepsilon} f^{i,\varepsilon}v_i^\varepsilon \sqrt{g^\varepsilon} dx^\varepsilon + \int_{\Gamma_+^\varepsilon\cup\Gamma_-^\varepsilon} h^{i,\varepsilon} v_i^\varepsilon\sqrt{g^\varepsilon}  d\Gamma^\varepsilon  \quad \forall \bv^\varepsilon\in V(\Omega^\varepsilon), \aes,
	\\\displaystyle \nonumber
	& \bu^\varepsilon(0,\cdot)= \bu_0^\varepsilon(\cdot),
	\end{align}
\end{problem}
where the functions
\begin{align}\label{TensorAeps}
& A^{ijkl,\varepsilon}:= \lambda g^{ij,\varepsilon}g^{kl,\varepsilon} + \mu(g^{ik,\varepsilon}g^{jl,\varepsilon} + g^{il,\varepsilon}g^{jk,\varepsilon} ), 
\\ \label{TensorBeps}
& B^{ijkl,\varepsilon}:= \theta g^{ij,\varepsilon}g^{kl,\varepsilon} + \frac{\rho}{2}(g^{ik,\varepsilon}g^{jl,\varepsilon} + g^{il,\varepsilon}g^{jk,\varepsilon} ), 
\end{align}
are the contravariant components of the three-dimensional elasticity and viscosity tensors, respectively. We assume that the Lam\'e coefficients   $\lambda\geq0, \mu>0$ and the viscosity coefficients $\theta\geq 0,\rho\geq 0$  are all independent of $\var$. Moreover, the terms
\begin{align*}
e^\varepsilon_{i||j}(\bu^\var):= \frac1{2}(u^\varepsilon_{i||j}+ u^\varepsilon_{j||i})=\frac1{2}(\d^\varepsilon_ju^\varepsilon_i + \d^\varepsilon_iu^\varepsilon_j) - \Gamma^{p,\varepsilon}_{ij}u^\varepsilon_p,
\end{align*}
designate the covariant components of the linearized strain tensor associated with the displacement field $\bUcal^\var$of the set $\bTheta(\bar{\Omega}^\varepsilon)$.  Moreover, $f^{i,\var}$ denotes the contravariant components of the volumic force densities, $h^{i,\var}$ denotes contravariant components of surface force densities and $\bu_0^\var$ denotes the initial `` displacements " (actually, the initial displacement is $\bUcal_0^\var:=(u_0^\var)_i\bg^{i,\var}$).

Note that the following additional relations are satisfied,
\begin{align}\nonumber
\Gamma^{3,\varepsilon}_{\alpha 3}=\Gamma^{p,\varepsilon}_{33}&=0  \ \textrm{in} \ \bar{\Omega}^\varepsilon, \\
\label{tensor_terminos_nulos}
A^{\alpha\beta\sigma 3,\varepsilon}=A^{\alpha 333,\varepsilon}=B^{\alpha\beta\sigma 3 , \varepsilon}&=B^{\alpha 333, \varepsilon}=0 \ \textrm{in} \ \bar{\Omega}^\varepsilon,
\end{align}
as a consequence of the definition of $\bTheta$ in (\ref{bTheta}). 

The existence and uniqueness of solution of the Problem \ref{problema_eps} for $\var>0$ small enough, established in the following theorem,  was proved in \cite{intro2} (see Theorem 4.7). 
\begin{theorem}\label{Thexistunic}
	Let $\Omega^\var$ be a domain in $\mathbb{R}^3$ defined previously in this section and let $\bTheta$ be a  $\mathcal{C}^2$-diffeomorphism of $\bar{\Omega}^\var$ in its image $\bTheta(\bar{\Omega}^\var)$, such that the three vectors $\bg_i^\var(\bx)=\d_i^\var\bTheta(\bx^\var)$ are linearly independent for all $\bx^\var\in\bar{\Omega}^\var$. Let $\Gamma_0^\var$ be a $d\Gamma^\var$-measurable subset of $\gamma\times [-\var,\var]$ 
	such that  $meas(\Gamma_0^\var)>0.$
	Let $\fb^{i,\var}\in L^{2}(0,T; L^2(\Omega^\var)) $, $\bh^{i,\var}\in L^{2}(0,T; L^2(\Gamma_1^\var))$, where $\Gamma_1^\var:= \Gamma_+^\var\cup\Gamma_-^\var$. Let  $\bu_0^\var\in V(\Omega^\var). $ Then, there exists a unique solution $\bu^\var=(u_i^\var):[0,T]\times\Omega^\var \rightarrow \mathbb{R}^3$ satisfying the Problem \ref{problema_eps}. Moreover, $\bu^\var\in H^{1}(0,T;V(\Omega^\var))$. In addition to that, if $\dot{\fb}^{i,\var}\in L^{2}(0,T; L^2(\Omega^\var)) $, $\dot{\bh}^{i,\var}\in L^{2}(0,T; L^2(\Gamma_1^\var))$, then $\bu^\var\in H^{2}(0,T;V(\Omega^\var))$.
\end{theorem}

\section{The scaled three-dimensional shell problem}\setcounter{equation}{0} \label{seccion_dominio_ind}

For convenience, we consider a reference domain independent of the small parameter $\var$. Hence, let us define the three-dimensional domain $\Omega:=\omega \times (-1, 1) $ and  its boundary $\Gamma=\d\Omega$. We also define the following parts of the boundary,
\begin{align*}
\Gamma_+:=\omega\times \{1\}, \quad \Gamma_-:= \omega\times \{-1\},\quad \Gamma_0:=\gamma_0\times[-1,1].
\end{align*}
Let $\bx=(x_1,x_2,x_3)$ be a generic point in $\bar{\Omega}$ and we consider the notation $\d_i$ for the partial derivative with respect to $x_i$. We define the following projection map, 
\begin{align*}
\pi^\varepsilon:\bx=(x_1,x_2,x_3)\in \bar{\Omega} \longrightarrow \pi^\varepsilon(\bx)=\bx^\varepsilon=(x_1^\var,x_2^\var,x_3^\var)=(x_1,x_2,\varepsilon x_3)\in \bar{\Omega}^\varepsilon,
\end{align*}
hence, $\d_\alpha^\varepsilon=\d_\alpha $  and $\d_3^\varepsilon=\frac1{\varepsilon}\d_3$. We consider the scaled unknown $\bu(\varepsilon)=(u_i(\varepsilon)):[0,T]\times \bar{\Omega}\longrightarrow \mathbb{R}^3$ and the scaled vector fields $\bv=(v_i):\bar{\Omega}\longrightarrow \mathbb{R}^3 $ defined as
\begin{align*}
u_i^\varepsilon(t,\bx^\varepsilon)=:u_i(\varepsilon)(t,\bx) \ \textrm{and} \ v_i^\varepsilon(\bx^\varepsilon)=:v_i(\bx) \ \forall \bx^\varepsilon=\pi^\varepsilon(\bx)\in \bar{\Omega}^\varepsilon, \ \forall \ t\in[0,T].
\end{align*}

Also, let the functions, $\Gamma_{ij}^{p,\varepsilon}, g^\varepsilon, A^{ijkl,\varepsilon}, B^{ijkl,\varepsilon}$ defined in (\ref{simbolos3D}), (\ref{g}), (\ref{TensorAeps}) and (\ref{TensorBeps}), be associated with the functions $\Gamma_{ij}^p(\varepsilon), g(\varepsilon), A^{ijkl}(\varepsilon), B^{ijkl}(\varepsilon)$ defined by
\begin{align} \label{escalado_simbolos}
&\Gamma_{ij}^p(\varepsilon)(\bx):=\Gamma_{ij}^{p,\varepsilon}(\bx^\varepsilon),
\\\label{escalado_g}
& g(\varepsilon)(\bx):=g^\varepsilon(\bx^\varepsilon),
\\\label{tensorA_escalado}
& A^{ijkl}(\varepsilon)(\bx):=A^{ijkl,\varepsilon}(\bx^\varepsilon),
\\\label{tensorB_escalado}
& B^{ijkl}(\varepsilon)(\bx):=B^{ijkl,\varepsilon}(\bx^\varepsilon),
\end{align}
for all $\bx^\varepsilon=\pi^\varepsilon(\bx)\in\bar{\Omega}^\varepsilon$. For all $\bv=(v_i)\in [H^1(\Omega)]^3$, let there be associated the scaled linearized strains components $\eij(\var;\bv)\in L^2(\Omega)$, defined by
\begin{align} \label{eab}
&\eab(\varepsilon;\bv):=\frac{1}{2}(\d_\beta v_\alpha + \d_\alpha v_\beta) - \Gamma_{\alpha\beta}^p(\varepsilon)v_p,\\ \label{eatres}
& \eatres(\varepsilon;\bv):=\frac{1}{2}\left(\frac{1}{\var}\d_3 v_\alpha + \d_\alpha v_3\right) - \Gamma_{\alpha 3}^p(\varepsilon)v_p,\\ \label{edtres}
& \edtres(\varepsilon;\bv):=\frac1{\varepsilon}\d_3v_3.
\end{align}
Note that with these definitions it is verified that
$
\eij^\var(\bv^\var)(\pi^\var(\bx))=\eij(\var;\bv)(\bx)$ $\forall\bx\in\Omega.
$

\begin{remark} The functions $\Gamma_{ij}^p(\varepsilon), g(\varepsilon), A^{ijkl}(\varepsilon), B^{ijkl}(\varepsilon)$ converge in $\mathcal{C}^0(\bar{\Omega})$ when $\varepsilon$ tends to zero.
\end{remark}

\begin{remark} 
	When we consider
	$\varepsilon=0$ the functions will be defined with respect to $\by\in\bar{\omega}$. Notice that (\ref{eatres}) and (\ref{edtres}) are not defined in that limit case, leading to a singular perturbation problem.  This fact motivates the use of asymptotic methods  for these kind of problems.
	
	Besides, we shall distinguish the three-dimensional Christoffel symbols from the two-dimensional ones by using  $\Gamma_{\alpha \beta}^\sigma(\varepsilon)$ and $ \Gamma_{\alpha\beta}^\sigma$, respectively.
\end{remark}

The next result is an adaptation of $(b)$ in Theorem 3.3-2, \cite{Ciarlet4b} to the viscoelastic case.  We will study the asymptotic behaviour of the scaled contravariant components $A^{ijkl}(\var), B^{ijkl}(\var)$ of the three-dimensional elasticity and viscosity tensors defined in (\ref{tensorA_escalado})--(\ref{tensorB_escalado}), as $\var\rightarrow0$.  We show their uniform positive definiteness  not only with respect to $\bx\in\bar{\Omega}$, but also with respect to $\var$, $0<\var\leq\var_0$. Finally, their limits are functions of $\by\in\bar{\omega}$ only, that is, independent of the transversal variable $x_3$.

\begin{theorem} \label{Th_comportamiento asintotico}
	Let $\omega$  be a domain in $\mathbb{R}^2$ and let $\btheta\in\mathcal{C}^2(\bar{\omega};\mathbb{R}^3)$ be an injective mapping such that the two vectors $\ba_\alpha=\d_\alpha\btheta$ are linearly independent at all points of $\bar{\omega}$, let $a^{\alpha\beta}$ denote the contravariant components of the metric tensor of $S=\btheta(\bar{\omega})$. In addition to that, let the other assumptions on the mapping $\btheta$ and the definition of $\var_0$ be as in Theorem \ref{var_0}. The contravariant components $A^{ijkl}(\var), B^{ijkl}(\var)$ of the scaled three-dimensional elasticity and viscosity tensors, respectively, defined in (\ref{tensorA_escalado})--(\ref{tensorB_escalado}) satisfy
	\begin{align*}
	A^{ijkl}(\var)= A^{ijkl}(0) + \mathcal{O}(\var) \ \textrm{and} \ A^{\alpha\beta\sigma 3}(\var)=A^{\alpha 3 3 3}(\var)=0, \\
	B^{ijkl}(\var)= B^{ijkl}(0) + \mathcal{O}(\var) \ \textrm{and} \ B^{\alpha\beta\sigma 3}(\var)=B^{\alpha 3 3 3}(\var)=0 ,
	\end{align*}
	for all $\var$, $0<\var \leq \var_0$, 
	and
	\begin{align*}
	A^{\alpha\beta\sigma\tau}(0)&= \lambda a^{\alpha\beta}a^{\sigma\tau} + \mu(a^{\alpha\sigma}a^{\beta\tau} + a^{\alpha\tau}a^{\beta\sigma}), & A^{\alpha\beta 3 3}(0)&= \lambda a^{\alpha\beta},
	\\
	A^{\alpha 3\sigma 3}(0)&=\mu a^{\alpha\sigma} ,& A^{33 3 3}(0)&= \lambda + 2\mu,
	\\
	A^{\alpha\beta\sigma 3}(0) &=A^{\alpha 333}(0)=0,
	\\
	B^{\alpha\beta\sigma\tau}(0)&= \theta a^{\alpha\beta}a^{\sigma\tau} + \frac{\rho}{2}(a^{\alpha\sigma}a^{\beta\tau} + a^{\alpha\tau}a^{\beta\sigma}),& B^{\alpha\beta 3 3}(0)&= \theta a^{\alpha\beta},
	\\
	B^{\alpha 3\sigma 3}(0)&=\frac{\rho}{2} a^{\alpha\sigma} ,& B^{33 3 3}(0)&= \theta + \rho, 
	\\
	B^{\alpha\beta\sigma 3}(0) &=B^{\alpha 333}(0)=0.
	\end{align*}
	
	Moreover, there exist two constants $C_e>0$ and $C_v>0$, independent of the variables and $\var$, such that 
	\begin{align} \label{elipticidadA_eps}
	\sum_{i,j}|t_{ij}|^2\leq C_e A^{ijkl}(\varepsilon)(\bx)t_{kl}t_{ij},\\\label{elipticidadB_eps}
	\sum_{i,j}|t_{ij}|^2 \leq C_v B^{ijkl}(\varepsilon)(\bx)t_{kl}t_{ij},
	\end{align}
	for all $\var$, $0<\var\leq\var_0$, for all $\bx\in\bar{\Omega}$ and all $\bt=(t_{ij})\in\mathbb{S}^3$.
\end{theorem}

\begin{remark}
	Note that the proof of the ellipticity for the scaled viscosity tensor $\left(B^{ijkl}(\varepsilon)\right)$ would follow the steps of the proof of the ellipticity for the elasticity tensor $\left(A^{ijkl}(\var)\right)$ in Theorem 3.3-2, \cite{Ciarlet4b}, since from a quality point of view their expressions differ in replacing the Lam\'e constants by the two viscosity coefficients. 
\end{remark}

Let the scaled applied forces $\bbf^i(\varepsilon):[0,T]\times \Omega\longrightarrow \mathbb{R}^3$ and  $\bbh^i(\varepsilon):[0,T]\times (\Gamma_+\cup\Gamma_-)\longrightarrow \mathbb{R}^3$ be defined by
\begin{align*}
\bbf^\var&=(f^{i,\varepsilon})(t,\bx^\varepsilon)=:\bbf(\var)= (f^i(\varepsilon))(t,\bx) 
\\ \nonumber
&\forall \bx\in\Omega, \ \textrm{where} \ \bx^\varepsilon=\pi^\varepsilon(\bx)\in \Omega^\varepsilon \ \textrm{and} \ \forall t\in[0,T], \\ 
\bbh^\var&=(h^{i,\varepsilon})(t,\bx^\varepsilon)=:\bbh(\var)= (h^i(\varepsilon))(t,\bx) 
\\ \nonumber 
&\forall \bx\in\Gamma_+\cup\Gamma_-, \ \textrm{where} \ \bx^\varepsilon=\pi^\varepsilon(\bx)\in \Gamma_+^\varepsilon\cup\Gamma_-^\varepsilon \ \textrm{and} \ \forall t\in[0,T].
\end{align*}
Also, we introduce $\bu_0(\var): \Omega \longrightarrow \mathbb{R}^3$ by
$
\bu_0(\var)(\bx):=\bu_0^\var(\bx^\var)$  $\forall$ $\bx\in\Omega$,  where  $\bx^\varepsilon=\pi^\varepsilon(\bx)\in \Omega^\varepsilon$ and define the space
\begin{align*} 
V(\Omega):=\{\bv=(v_i)\in [H^1(\Omega)]^3; \bv=\mathbf{0} \ on \ \Gamma_0\},
\end{align*}
which is a Hilbert space, with associated norm denoted by $\left\| \cdot\right\| _{1,\Omega}$.

We assume that the scaled applied forces are given by
\begin{align*} 
& \bbf(\varepsilon)(t, \bx)=\varepsilon^p\bbf^p(t,\bx) \ \forall \bx\in \Omega \ \textrm{and} \ \forall t\in[0,T], \\ 
& \bbh(\varepsilon)(t, \bx)=\varepsilon^{p+1}\bbh^{p+1}(t,\bx) \ \forall  \bx\in \Gamma_+\cup\Gamma_- \ \textrm{and} \ \forall t\in[0,T],
\end{align*}
where $\bbf^p$ and $\bbh^{p+1}$ are functions independent of $\var$ and where $p$ is a natural number that will show the order of the volume and surface forces, respectively. Then, the scaled variational problem  can   be written as follows:
\begin{problem}\label{problema_orden_fuerzas}
	Find $\bu(\varepsilon):[0,T]\times\Omega\longrightarrow \mathbb{R}^3$ such that,
	\begin{align} \nonumber
	& \bu(\varepsilon)(t,\cdot)\in V(\Omega) \forallt, \\ \nonumber
	&\int_{\Omega}A^{ijkl}(\varepsilon)e_{k||l}(\varepsilon;\bu(\varepsilon))e_{i||j}(\varepsilon;\bv)\sqrt{g(\varepsilon)} dx
	+ \int_{\Omega} B^{ijkl}(\varepsilon)e_{k||l}(\varepsilon;\dot{\bu}(\varepsilon))e_{i||j}(\varepsilon;\bv) \sqrt{g(\varepsilon)}  dx
	\\ \nonumber 
	&\quad= \int_{\Omega} \var^p \fb^{i,p}v_i \sqrt{g(\varepsilon)} dx + \int_{\Gamma_+\cup\Gamma_-} \var^{p}\bh^{i,p+1} v_i\sqrt{g(\varepsilon)}  d\Gamma 
	\quad \forall \bv\in V(\Omega), \aes,
	\\\displaystyle \nonumber
	& \bu(\var)(0,\cdot)= \bu_0(\var)(\cdot).
	\end{align} 
\end{problem}
From now on,   for each $\var>0$, we shall use  the shorter notation $\eij(\var)\equiv\eij(\varepsilon;\bu(\varepsilon))$ and $\deij(\var)\equiv\eij(\varepsilon;\dot{\bu}(\varepsilon))$, for its time derivative. 
Analogously to Theorem \ref{Thexistunic}, we can prove the existence of an unique solution $\bu(\var)\in H^{1}(0,T;V(\Omega))$ (or $\bu(\var)\in H^{2}(0,T;V(\Omega))$, respectively) of the Problem \ref{problema_orden_fuerzas}  (see Theorem 4.7, \cite{intro2}) for each $\var>0$.


\section{Technical preliminaries}\setcounter{equation}{0} \label{preliminares}

Concerning geometrical and mechanical preliminaries, we shall present some theorems, which will be used in the following sections. 
First, we recall the Theorem 3.3-1, \cite{Ciarlet4b}.

\begin{theorem} \label{Th_simbolos2D_3D}
	Let $\omega$ be a domain in $\mathbb{R}^2$, let $\btheta\in\mathcal{C}^3(\bar{\omega};\mathcal{R}^3)$ be an injective mapping such that the two vectors $\ba_\alpha=\d_\alpha\btheta$ are linearly independent at all points of $\bar{\omega}$ and let $\var_0>0$ be as in Theorem \ref{var_0}. The functions $\Gamma^p_{ij}(\var)=\Gamma^p_{ji}(\var)$ and $g(\var)$ are defined in (\ref{escalado_simbolos})--(\ref{escalado_g}), the functions $b_{\alpha\beta}, b_\alpha^\sigma, \Gamma_{\alpha\beta}^\sigma,a$, are defined in Section \ref{problema} and the covariant derivatives $b_\beta^\sigma|_\alpha$ are defined by
	\begin{align} \nonumber 
	b_\beta^\sigma|_\alpha:=\d_\alpha b_\beta^\sigma +\Gamma^\sigma_{\alpha\tau}b_\beta^\tau - \Gamma^\tau_{\alpha\beta}b^\sigma_\tau.
	\end{align}
	The functions $b_{\alpha\beta}, b_\alpha^\sigma, \Gamma_{\alpha\beta}^\sigma, b_\beta^\sigma|_\alpha$ and $a$ are identified with functions in $\mathcal{C}^0(\bar{\Omega})$. Then
	\begin{align*}
	\begin{aligned}[c]
	\Gamma_{\alpha\beta}^\sigma(\var)&=  \Gamma_{\alpha\beta}^\sigma -\var x_3b_\beta^\sigma|_\alpha + \mathcal{O}(\var^2), \\
	\d_3 \Gamma_{\alpha\beta}^p(\var)&= \mathcal{O}(\var), 
	\\
	\Gamma_{\alpha3}^3(\var)&=\Gamma_{33}^p(\var)=0,
	\end{aligned}
	\qquad
	\begin{aligned}[c]
	\Gamma_{\alpha\beta}^3(\var)&=b_{\alpha\beta} - \var x_3 b_\alpha^\sigma b_{\sigma\beta}, 
	\\
	\Gamma_{\alpha3}^\sigma(\var)& = -b_\alpha^\sigma - \var x_3 b_\alpha^\tau b_\tau^\sigma + \mathcal{O}(\var^2), 
	\\
	g(\varepsilon)&=a + \mathcal{O}(\varepsilon),
	\end{aligned}
	\end{align*}
	for all $\var$, $0<\var\leq\var_0$, where the order symbols $\mathcal{O}(\var)$ and $\mathcal{O}(\var^2)$  are meant with respect to the norm $\left\| \cdot\right\| _{0,\infty,\bar{\Omega}}$ defined by
	$
	\left\| w\right\| _{0,\infty,\bar{\Omega}}=\sup \{|w(\bx)|; \bx\in\bar{\Omega}\}.
	$
	Finally, there exist constants $a_0, g_0$ and $g_1$ such that
	\begin{align} \nonumber
	& 0<a_0\leq a(\by) \ \forall \by\in \bar{\omega},
	\\ \label{g_acotado}
	& 0<g_0\leq g(\varepsilon)(\bx) \leq g_1 \ \forall \bx\in\bar{\Omega} \ \textrm{and} \ \forall \ \var, 0<\varepsilon\leq \varepsilon_0.
	\end{align}
\end{theorem}


We now include the following result that will be used repeatedly in what follows (see Theorem 3.4-1, \cite{Ciarlet4b}, for details).

\begin{theorem} \label{th_int_nula}
	Let $\omega$ be a domain in $\mathbb{R}^2$ with boundary $\gamma$, let $\Omega=\omega\times (-1,1)$, and let $g\in L^p(\Omega)$, $p>1$, be a function such that 
	\begin{align*}
	\intO g \d_3v dx=0, \ \textrm{for all} \ v\in \mathcal{C}^{\infty}(\bar{\Omega}) \ \textrm{with} \ v=0 \on \gamma\times[-1,1]. 
	\end{align*}
	Then $g=0.$
\end{theorem}
\begin{remark}
	This result holds if $\intO g \d_3v dx=0$ for all $v\in H^1(\Omega)$ such that $v=0$ in $\Gamma_0$. We will use this result in this way in what follows.
\end{remark}


We now introduce the average with respect to the transversal variable, which plays a major role in this study. To that end, let $\bv$ represent real or vectorial functions defined  almost everywhere over $\Omega=\omega\times (-1,1)$. We define the transversal average by
\begin{align*}
\bar{\bv}(\by)=\frac1{2}\int_{-1}^{1}\bv(\by,x_3)dx_3,
\end{align*}
for almost all $\by\in\omega$.
Given $\beeta=(\eta_i)\in [H^1(\omega)]^3,$ let
\begin{align} \label{def_gab}
\gab(\beeta):= \frac{1}{2}(\d_\beta\eta_\alpha + \d_\alpha\eta_\beta) - \Gamma_{\alpha\beta}^\sigma\eta_\sigma -  b_{\alpha\beta}\eta_3,
\end{align}
denote the covariant components of the linearized change of metric tensor associated with a displacement field $\eta_i\ba^i$ of the surface $S$.
In the next theorem we introduce some results related with the transversal averages that will be useful in what follows.  
\begin{theorem}\label{Th_medias}
	Let $\omega$ be a domain in $\mathbb{R}^2$, let $\Omega=\omega\times(-1,1)$ and $T>0$.
	\begin{enumerate}[label={{(\alph*)}}, leftmargin=0em ,itemindent=3em]
		\item Let $v\in\WLO$. Then $\bar{v}(\by)$ is finite for almost all  $\by\in\omega$,  belongs to $\WLo$, and
		\begin{align}\nonumber
		\left|\bar{v}\right|_{\WLo}\leq\frac{1}{\sqrt{2}}\left| v\right| _{\WLO}.
		\end{align}
		If $\d_3v=0$ in the distributions sense $\left(\int_{\Omega}v \d_3{\varphi} dx=0 \ \forall {\varphi} \in \mathcal{D}(\Omega) \right)$ then $v$ does not depend on $x_3$ and
		\begin{align}\nonumber
		v(\by,x_3)=\bar{v}(\by) \ \textrm{for almost all } \ (\by,x_3)\in\Omega.
		\end{align}
		\item Let $v\in \WHO$. Then $\bar{v}\in \WHo$, $\d_\alpha\bar{v}=\overline{\d_\alpha v}$ and
		\begin{align}\nonumber
		\left\| \bar{v}\right\| _{\WHo}\leq\frac{1}{\sqrt{2}}\left\| v\right\| _{\WHO}.
		\end{align}
		Let $\gamma_0$ be a subset $\d\gamma$-measurable of $\gamma$. If $v=0$ on $\gamma_0\times[-1,1]$ then $\bar{v}=0$ on $\gamma_0$; in particular, $\bar{v}\in H^{1}(0,T;H^1_0(\omega))$ if $v=0$ on $\gamma\times[-1,1]$.
	\end{enumerate}
\end{theorem}
Now, we shall introduce two results that will be needed for the convergence result. Given $\bv=(v_i)\in [H^1(\Omega)]^3 $ let us define:
\begin{align*} 
\gab(\bv):=& \frac{1}{2}(\d_\beta v_\alpha + \d_\alpha v_\beta) - \Gamma_{\alpha\beta}^\sigma v_\sigma -  b_{\alpha\beta} v_3,
\\
\rho_{\alpha\beta}(\bv):=& \d_{\alpha\beta}v_3 - \Gamma_{\alpha\beta}^\sigma \d_\sigma v_3 - b_\alpha^\sigma b_{\sigma\beta} v_3 + b_\alpha^\sigma (\d_\beta v_\sigma- \Gamma_{\beta\sigma}^\tau v_\tau) 
\\ \qquad &+ b_\beta^\tau(\d_\alpha v_\tau-\Gamma_{\alpha\tau}^\sigma v_\sigma ) + b^\tau_{\beta|\alpha} v_\tau,
\\
\eab^1(\var;\bv):=& \frac{1}{\var}\gab(\bv) + x_3 b_{\beta|\alpha}^\sigma v_\sigma + x_3b_\alpha^\sigma b_{\sigma\beta} v_3.
\end{align*}

\begin{theorem}\label{Th_521}
	Let the functions $\Gamma^\sigma_{\alpha\beta}, b_{\alpha\beta}, b_\alpha^\beta\in \mathcal{C}^0(\bar{\omega})$ be identified with functions in $\mathcal{C}^0(\bar{\Omega})$ and we consider $\var_0$ defined as in Theorem \ref{var_0}. Then there exists a constant $\tilde{C}>0$ such that for all $\var$, $0<\var\leq \var_0$ and all $\bv\in H^{1}(0,T; [H^1(\Omega)]^3)$, the scaled linearized strains $\eab(\var;\bv)$ satisfy:
	\begin{align*}
	\left|\frac{1}{\var}\eab(\var;\bv) - \eab^1(\var;\bv)\right|_{\WLO}&\leq \tilde{C} \var \sum_{\alpha}\left| v_\alpha\right| _{\WLO},
	\\
	\left\|\frac{1}{\var}\d_3\eab(\var;\bv)+ \rab(\bv)\right\|_{\WHMO}&\leq \tilde{C}\left(\sum_i|\eitres(\var;\bv)|_{\WLO} \right.
	\\
	+ \var  \sum_\alpha|v_\alpha|_{\WLO} &  + \var\left\| v_3\right\| _{\WHO}{\Bigg) }   .
	\end{align*}
\end{theorem}   

\begin{theorem}\label{Th_522}
	Let $(\bu(\var))_{\var>0}$ be a sequence of functions $\bu(\var)\in\WVO$ that satisfies
	\begin{align*}
	\bu(\var)&\deb\bu \en \WHOt,
	\\
	\frac{1}{\var}\eij(\var;\bu(\var))&\deb \eij^1 \en \WLO,
	\end{align*}
	when $\var\to0.$ Then,
	\begin{enumerate}[label={{(\alph*)}}]
		\item  $\bu$ is independent of the transversal variable $x_3$. 
		\item  $\bar{\bu}\in H^1(\omega)\times H^1(\omega) \times H^2(\omega)$ with $\bar{u}_i=\d_\nu \bar{u}_3=0$ on $\gamma_0$.
		\item $\gab(\bu)=0$.
		\item  $\rab(\bu)\in \WLO$ and $\rab(\bu)=-\d_3\eab^1.$
		
		\item If in addition, there exist functions $\kappa_{\alpha\beta}\in\WHMO$ such that $\rab(\bu(\var))\to\kappa_{\alpha\beta}$ in $\WHMO$ as $\var\to 0$, then
		\begin{align*}
		\bu(\var)&\to \bu \en \WHOt,\\
		\rab(\bu)&=\kappa_{\alpha\beta} \ \textrm{hence,} \ \kappa_{\alpha\beta}\in\WLO.
		\end{align*}
		
	\end{enumerate}
	
\end{theorem}

\begin{remark}
	Theorems \ref{Th_medias}, \ref{Th_521} and \ref{Th_522} are  generalizations of Theorems 4.2-1, 5.2-1 and 5.2-2, \cite{Ciarlet4b}, respectively, and their proofs follow straightforward from the results presented there. 
\end{remark}

Finally, in the next theorem we recall a three-dimensional inequality of Korn's type for a family of viscoelastic shells (see Theorem 5.3-1, \cite{Ciarlet4b}).

\begin{theorem}\label{Th_desigKorn}
	Assume that $\btheta \in \mathcal{C}^3(\bar{\omega};\mathbb{R}^3)$ and we consider $\var_0$ defined as in Theorem \ref{var_0}. We consider a family of viscoelastic membrane shells with thickness $2\varepsilon$ with each having the same  middle surface $S=\btheta(\bar{\omega})$ and with each subjected to a boundary condition of place along a portion of its lateral face having the same set $\btheta(\gamma_0)$ as its middle curve.
	Then there exist a constant $\varepsilon_1$ verifying $0<\varepsilon_1<\varepsilon_0$ and a constant $C>0$ such that, for all $\var$, $0<\varepsilon\leq \varepsilon_1$, the following three-dimensional inequality of Korn's type holds,
	\begin{equation}\label{Korn}
	\left\| \bv\right\| _{1,\Omega}\leq \frac{C}{\var}\left(\sum_{i,j}|\eij(\varepsilon;\bv)|^2_{0,\Omega}\right)^{1/2} \ \forall \bv=(v_i)\in V(\Omega).
	\end{equation}
\end{theorem}

\section{Completion spaces and Admissible forces}\setcounter{equation}{0} \label{seccion_fuerzas_admisibles}


In this section we shall introduce {\it ad hoc} spaces which complete the ones introduced in \cite{intro2}, where the obtention of the two-dimensional equations of the viscoelastic membrane shell problem was presented.  Moreover, we also shall introduce some assumptions needed on the applied forces. Recall that,
\begin{align*}
V_F(\omega):=& \{ \beeta=(\eta_i) \in H^1(\omega)\times H^1(\omega)\times H^2(\omega);
\\ &\quad  \eta_i=\d_\nu \eta_3=0 \ \textrm{on} \ \gamma_0, \gab(\beeta)=0 \en \omega \}.
\end{align*}  
In \cite{eliptico}, we justified the two-dimensional equations of  the  viscoelastic membrane shells, where the middle surface $S$ is elliptic and the boundary condition of place is considered on the whole lateral face of the shell. These assumptions lead to $V_F(\omega)=\{\bcero\}$ (see \cite{eliptico} for details). In this paper, we shall considered the remaining cases where some of those assumptions are not verified but still $V_F(\omega)=\{\bcero\}$. Those cases are known as the generalized membrane shells. Let us define the spaces :
\begin{align*}
V(\omega)&:=\{\beeta=(\eta_i)\in[H^1(\omega)]^3 ; \eta_i=0 \ \textrm{on} \ \gamma_0 \}, \\ 
V_0(\omega)&:=\{\beeta=(\eta_i)\in V(\omega) ; \gab(\beeta)=0  \ \textrm{in} \ \omega \},\\
V_K(\omega)&:=\{ \beeta=(\eta_i)\in H^1(\omega)\times H^1(\omega)\times H^2(\omega); \eta_i=\d_\nu\eta_3=0 \ \textrm{on} \ \gamma_0 \},
\end{align*}  
and also, we introduce  the seminorms defined by
\begin{align*}
|\bv|^M_\Omega&:= \left(\left| \d_3\bv\right| ^2_{0,\Omega} + \left(\left| \bar{\bv}\right| _\omega^M\right)^2\right) ^{1/2} \ \forall \bv\in V(\Omega), \\
\left|\beeta\right|_\omega^M&:= \left(\sum_{\alpha,\beta} \left| \gab(\eta)\right|_{0,\omega}^2 \right)^{1/2}  \forall \beeta=(\eta_i)\in H^1(\omega)\times H^1(\omega)\times L^2(\omega).
\end{align*}
Since $V_F(\omega)=\{\bcero\}$ by assumption, the seminorm  $\left|\cdot\right|_\omega^M$ 
is a norm over the space $V_K(\omega).$

Now, we shall distinguish two different subsets of generalized membrane shells, depending on whether or not the space $V_0(\omega)$ contains only the zero function. One of the difficulties faced is the introduction of abstract spaces, which do not have any physical meaning.
We consider a generalized membrane shell of the first kind when $V_0(\omega)=\{\bcero\}$ (hence, $V_F(\omega)=\{\bcero\}$), this is, when the seminorm $\left|\cdot\right| _{\omega}^M$ is a norm over the space $V(\omega)$ (hence, will be a norm over $V_K(\omega)\subset V(\omega)$). Therefore, the abstract spaces are defined by
\begin{align} \label{espacio_Omega_1tipo}
V_M^{\#}(\Omega)&:= \textrm{completion of} \ V(\Omega) \ \textrm{with respect} \ |.|^M_\Omega,\\ \label{espacio_omega1tipo}
V_M^{\#}(\omega)&:= \textrm{completion of} \ V(\omega) \ \textrm{with respect} \ |.|^M_\omega.
\end{align}

Otherwise, if $V_0(\omega)\neq\{\bcero\}$ but still $V_F(\omega)=\{\bcero\}$ , this is, if $\left|\cdot\right| _\omega^M$ is a norm over $V_K(\omega)$ but not over $V(\omega)$, the shell is a generalized membrane of the second kind.  Therefore, the abstract spaces are defined by
\begin{align}
\tilde{V}_M^{\#}(\Omega)&:= \textrm{completion of} \ V(\Omega)/V_0(\Omega) \ \textrm{with respect to} \ |\cdot|^M_\Omega,
\\
\tilde{V}_M^{\#}(\omega)&:= \textrm{completion of} \ V(\omega)/V_0(\omega) \ \textrm{with respect to} \ |\cdot|^M_\omega.
\end{align}

\begin{remark}
	Notice that, in both cases, these ``abstract'' spaces might not be spaces of distributions. 
	
	We can find a large variety of practical examples in the case of generalized membranes of the first kind (see \cite{Ciarlet4b}). However, we do not have examples for those of the second kind. As commented in \cite{Ciarlet4b}, they should correspond to shells with surfaces $S$ with ``few" regularity.  
\end{remark}

Now, we shall present some additional assumptions needed for the applied forces. Let us define for each $\var>0$, the real function  $L(\varepsilon)(t): V(\Omega)\longrightarrow \mathbb{R}$ given by
\begin{equation}
L(\varepsilon)(t)(\bv):= \int_{\Omega} f^{i}(t) v_i \sqrt{g(\varepsilon)} dx + \int_{\Gamma_+\cup\Gamma_-} h^{i}(t) v_i\sqrt{g(\varepsilon)}  d\Gamma,  
\end{equation} 
$\forall \bv\in V(\Omega), \forallt,$ with $f^{i}\in L^2(0,T;L^2(\Omega))$ and $h^{i}\in L^2(0,T;L^2(\Gamma_+\cup\Gamma_-))$. It is easy to check that this function is continuous with respect to the norm $||.||_{1,\Omega}$ and uniform with respect  $0<\varepsilon\leq\varepsilon_0$, with $\var_0$ defined in Theorem \ref{var_0}. By the inequality of Korn's type in Theorem \ref{Th_desigKorn}, there exists a constant $K$ such that
\begin{equation}
\nonumber
|L(\varepsilon)(t)(\bv)|\leq \frac{K(t)}{\varepsilon} \left(\sum_{i,j}| \eij(\varepsilon;\bv)|_{0,\Omega}^2   \right) ^{1/2} \ \forall \bv\in V(\Omega), \forallt.
\end{equation}
Therefore,  $L(\varepsilon)(t)$ is also continuous with respect to the norm defined by 
\begin{align}\label{norma_e}
\bv\longmapsto \left(\sum_{i,j}| \eij(\varepsilon;\bv)|_{0,\Omega}^2\right)^{1/2},
\end{align}  
but not uniform whit respect to  $\varepsilon$ unless additional hypothesis for the applied forces is made. Notice that    $V(\Omega)$ is a Hilbert space with respect to the interior product,
\begin{align}
\left(\bv,\bw\right):= \intO \eij(\var;\bv) \eij(\var;\bw) \sqrt{g(\var)} dx, \quad \forall\bv,\bw\in V(\Omega),
\end{align}
since it is easy to verify that the norm  (\ref{norma_e}) satisfies the parallelogram's equality. Then, applying the Riez's Representation  Theorem, there exists a  $\bG(t)\in [H^1(\Omega)]^3$ for each $t\in[0,T]$, such that
\begin{equation} 
L(\varepsilon)(t)(\bv)=\int_{\Omega} 
\eij(\var; \bG(t))\eij(\varepsilon;\bv)\sqrt{g(\varepsilon)}dx \ \forall \bv\in V(\Omega).
\end{equation} 
Therefore, let us define $F^{ij}(\var)(t):=\eij(\var; \bG(t))$ for each $t\in[0,T]$, so
\begin{equation} \label{fuerzasadmisibles_def}
L(\varepsilon)(t)(\bv)=\int_{\Omega} F^{ij}(\varepsilon)(t)\eij(\varepsilon;\bv)\sqrt{g(\varepsilon)}dx \ \forall \bv\in V(\Omega).
\end{equation} 
If $|F^{ij}(\varepsilon)|_{0,\Omega}$ is uniformly bounded with respect to  $\varepsilon$, we ensure the uniform continuity of the linear form. Moreover, we need $F^{ij}(\varepsilon)$ to have a limit in  $\LLO$,  when $\varepsilon\rightarrow 0$. Following the considerations above, the applied forces over a family of generalized membranes will be known as admissible forces if, for each $\varepsilon>0$, there exist functions  $F^{ij}(\varepsilon)=F^{ji}(\varepsilon)\in L^{2}(0,T;L^2(\Omega))$ and  $F^{ij}=F^{ji}\in L^{2}(0,T;L^2(\Omega)) $ such that the equality (\ref{fuerzasadmisibles_def}) holds for all $\varepsilon$, $ 0<\varepsilon\leq\varepsilon_0$ and $F^{ij}(\var)\rightarrow F^{ij}$ in $\LLO$  when $\varepsilon\rightarrow 0$. Therefore, if the applied forces are admissible, there exists a constant $K_0(t)>0$ such that,
\begin{equation}\label{fuerzasadmisibles}
\left|L(\varepsilon)(t)(\bv)\right|\leq {K_0(t)} \left(\sum_{i,j}| \eij(\varepsilon;\bv)|_{0,\Omega}^2 \right)   ^{1/2}, 
\end{equation}

We need to assume additional hypotheses for the contravariant components $f^{i,\varepsilon}\in L^{2}(0,T;L^2(\Omega))$, $h^{i,\var}\in L^{2}(0,T;L^2(\Gamma_+\cup\Gamma_-))$ so that, the right-hand side of the equation (\ref{Pbvariacionaleps}) can be written for each  $\varepsilon>0$ and for all $t\in[0,T]$ as follows:
\begin{equation}
\int_{\Omega^\varepsilon} f^{i,\varepsilon}(t) v_i^\varepsilon \sqrt{g^\varepsilon} dx^\varepsilon + \int_{\Gamma_+^\varepsilon\cup\Gamma_-^\varepsilon} h^{i,\varepsilon}(t) v_i^\varepsilon\sqrt{g^\varepsilon}  d\Gamma^\varepsilon= \var L(\varepsilon)(\bv)(t),
\end{equation} 
\begin{remark}
	Notice that, by considering this expression we are making an assumption on the orders of the applied forces. Actually, these orders are those corresponding to the viscoelastic membrane shell equations derived in \cite{intro2}, that is, taking $p=0$ in the Problem \ref{problema_orden_fuerzas}.
\end{remark}
Then, we can write the equations in the reference domain by taking into account the definition of the admissible forces.
\begin{problem}\label{problema_escalado_fuerzas}
	Find $\bu(\varepsilon):(0,T)\times \Omega\longrightarrow \mathbb{R}^3$ such that  
	\begin{align}\nonumber 
	&\bu(\varepsilon)(t,\cdot)\in V(\Omega) \forallt, \\ \nonumber
	&\int_{\Omega}A^{ijkl}(\varepsilon)e_{k||l}(\varepsilon)e_{i||j}(\varepsilon;\bv)\sqrt{g(\varepsilon)} dx
	+ \int_{\Omega} B^{ijkl}(\varepsilon)\dot{e}_{k||l}(\varepsilon)e_{i||j}(\varepsilon;\bv) \sqrt{g(\varepsilon)}  dx
	\\\label{ecuacion_fadmisibles}
	& \quad=  L(\varepsilon)(v)  \quad \forall \bv\in V(\Omega), \aes,
	\\\displaystyle \nonumber
	& \bu(\var)(0,\cdot)= \bu_0(\var)(\cdot).
	\end{align} 
\end{problem}

The Problem \ref{problema_escalado_fuerzas} is a particular case of the Problem \ref{problema_orden_fuerzas}, hence we can ensure the existence, uniqueness and regularity of solution for $\var$ sufficiently small, taking into account the admissible forces defined above.

\section{Asymptotic Analysis. Convergence results as $\var\to0$}\setcounter{equation}{0} \label{seccion_convergencia}

To begin with, we  recall the two-dimensional membrane shell problem obtained in \cite{intro2} taking into account the admissible forces and the abstract spaces defined in the previous section.  Let us remind the definition of the two-dimensional fourth-order tensors that appeared naturally in that study,
\begin{align} \label{tensor_a_bidimensional}
\a&:=\frac{2\lambda\rho^2 + 4\mu\theta^2}{(\theta + \rho)^2}a^{\alpha\beta}a^{\sigma\tau} + 2\mu\ten, 
\\ \label{tensor_b_bidimensional}
\b&:=\frac{2\theta\rho}{\theta + \rho}a^{\alpha\beta}a^{\sigma\tau} + \rho\ten,
\\ \label{tensor_c_bidimensional}
\c&:=\frac{2 \left(\theta \Lambda \right)^2}{\theta + \rho} a^{\alpha\beta}a^{\sigma\tau},  
\end{align}
where
\begin{align}\label{Lambda}
\Lambda:=\left(  \frac{\lambda}{\theta} - \frac{\lambda+ 2 \mu}{\theta + \rho} \right). 
\end{align}

For the sake of briefness, we only consider viscoelastic generalized membrane shells of the first kind, as those of the second kind are treated in a similar fashion. We  formulate the scaled two-dimensional variational problem of a viscoelastic generalized membrane shell of the first kind as follows:
\begin{problem}\label{problema_primertipo}
	Find $\bxi(t,\cdot)\in V_M^{\#}(\omega) \forallt $ such that,
	\begin{align}
	&B_{M}^{\#}(\bxi(t),\beeta)=L_{M}^{\#}(\beeta)(t) \ \forall \beeta\in V_M^{\#}(\omega), \ae,
	\\ \nonumber
	&\bxi(0,\cdot)=\bxi_0(\cdot),
	\end{align}
	where $B_{M}^{\#}$ and $L_{M}^{\#}$ are the unique continuous extensions from $H^{1}(0,T;V(\omega))$ to $H^{1}(0,T; V_M^{\#}(\omega))$ and from $V(\omega)$ to $V_M^{\#}(\omega)$ of the functions $B_{M}:H^{1}(0,T; V(\omega)) \times V(\omega) \longrightarrow \mathbb{R}$ and $L_{M}(t): V(\omega)\longrightarrow \mathbb{R}$, respectively, defined by
	\begin{align} \nonumber
	B_{M}(\bxi(t),\beeta):=&\int_{\omega}\a\gst(\bxi(t))\gab(\beeta)\sqrt{a}dy + \int_{\omega}\b\gst(\dot{\bxi}(t))\gab(\beeta)\sqrt{a}dy 
	\\
	\qquad & - \int_0^te^{-k(t-s)}\into \c \gst(\bxi(s))\gab(\beeta)\sqrt{a}dyds,
	\\
	L_{M}(\beeta)(t):=& \int_{\omega} \varphi^{\alpha\beta}(t)\gab(\beeta)\sqrt{a}dy , 
	\end{align}                    
	where we introduced the constant $k$ defined by
	\begin{align}\label{k}
	k:=\frac{\lambda+ 2 \mu}{\theta + \rho},
	\end{align}
	and where $\varphi^{\alpha\beta}$ is an auxiliary function, related with the admissible forces, that will appear naturally in this study,  given by
	\begin{align}\label{phi_1}
	\varphi^{\alpha\beta}(t):=\int_{-1}^{1}\left( F^{\alpha\beta}(t) - \frac{\theta}{\theta + \rho} F^{33}(t)a^{\alpha\beta} + \frac{\theta\Lambda}{\theta + \rho}\int_0^t e^{-k(t-s)}F^{33}(s)dsa^{\alpha\beta}\right) dx_3,
	\end{align}
	$\Forallt$.
\end{problem}

The Problem \ref{problema_primertipo} is well posed and it has a unique solution. The proof given in the next theorem makes use of similar arguments which can be found in the proof of the existence and uniqueness of solution of the de-scaled problem of the viscoelastic membrane shell (see Theorem 6.4,  \cite{intro2}).

\begin{theorem} \label{Th_exist_unic_bid_cero}
	Let $\omega$  be a domain in $\mathbb{R}^2$, let $\btheta\in\mathcal{C}^2(\bar{\omega};\mathbb{R}^3)$ be an injective mapping such that the two vectors $\ba_\alpha=\d_\alpha\btheta$ are linearly independent at all points of $\bar{\omega}$. Let $\varphi^{\alpha\beta}\in L^{2}(0,T; L^2(\omega)) $ and  $\bxi_0\in V_M^{\#}(\omega). $  Then the Problem \ref{problema_primertipo}, has a unique solution  $\bxi\in H^{1}(0,T;V_M^{\#}(\omega))$.  In addition to that, if $\dot{\varphi}^{\alpha\beta}\in L^{2}(0,T; L^2(\omega)) $, then $\bxi\in H^{2}(0,T;V_M^{\#}(\omega))$. 
\end{theorem}

For each $\var>0$, we assume that the initial condition for the scaled linear strains is
\begin{equation} \label{condicion_inicial_def}
\eij(\var)(0,\cdot)=0,
\end{equation}
this is, that the domain is on its natural state with no strains on it at the beginning of the period of observation.

Now, we present here the main result of this paper, that the scaled three-dimensional unknown, $\bu(\var)$, converges as $\var$ tends to zero towards a limit $\bu$. Moreover, its transversal  average, $\overline{\bu(\var)}$, converges as $\var$ tends to zero to the solution $\bxi=\bar{\bu}$ of the two-dimensional Problem \ref{problema_primertipo}, posed over the set $\omega$. Given $\bv\in L^2(0,T;[L^2(\Omega)]^3)$ and $\beeta\in L^2(0,T;[L^2(\omega)]^3)$, we shall use the notation
\begin{align*}
\left|\bv\right| _{T,\Omega}^M:= \left(\int_0^T \left( \left| \bv(t)\right| _\Omega^M \right)^2dt\right)^{1/2}, \quad \left|\beeta\right| _{T,\omega}^M:= \left(\int_0^T \left( \left| \beeta(t)\right| _\omega^M \right)^2dt\right)^{1/2}.
\end{align*}

\begin{theorem} \label{Th_convergencia}
	Let us suppose that
	$\btheta\in\mathcal{C}^3(\bar{\omega};\mathbb{R}^3)$ and let $\var_0$ be defined as in Theorem \ref{var_0}. Consider a family of generalized membrane shells of the first kind with thickness  $2\varepsilon$,  having each one the same middle surface $S=\btheta(\bar{\omega})$, under a boundary condition of place along a portion of  its lateral face, with the same set $\btheta(\gamma_0)$  as the middle curve and subjected to admissible forces. Let $\bu(\varepsilon)$ be for every $\var$, $0<\varepsilon\leq\varepsilon_0$ the solution of the three-dimensional problem under admissible forces in Problem \ref{problema_escalado_fuerzas} . Then there exists $\bu\in \WVMO$ and $\bxi\in \WVMo$ such that $\bu(\varepsilon)\rightarrow \bu$ in $\WVMO$ when $\varepsilon\rightarrow 0$. Moreover,
	\begin{align*}
	\overline{\bu(\varepsilon)}:= \frac1{2}\int_{-1}^{1}\bu(\varepsilon)dx_3\rightarrow \bxi \ \textrm{in} \ \WVMo \ \textrm{when} \ \varepsilon \rightarrow 0.
	\end{align*}
	Furthermore, the limit $\bxi$ satisfies the Problem \ref{problema_primertipo}.
\end{theorem}
\begin{proof} 
	We follow the same structure of the proof of the Theorem 5.6-1, \cite{Ciarlet4b}.Hence, we shall reference to some steps which apply in the same manner. The proof is divided into several parts, numbered from $(i)$ to $(xi)$.
	\begin{enumerate}[label={{(\roman*)}}, leftmargin=0em ,itemindent=3em]
		\item{\em There exists  $\var_2,$ $0<\varepsilon_2\leq \varepsilon_0$ and a constant $c_0>0$ such that, for all $0<\varepsilon\leq \varepsilon_2$,
			\begin{equation}
			\left|\bv\right| _\Omega^M\leq c_0 \left(\sum_{i,j}\left|\eij(\varepsilon;\bv)\right| _{0,\Omega}^{2}\right) ^{1/2} \ \forall \bv\in V(\Omega).
			\end{equation}}
		The proof can be found in step $(i)$ in Theorem 5.6-1, \cite{Ciarlet4b}, so we omit it.
		
		\item {\em \textit{A priori} boundedness and extractions of weakly convergent sequences. The seminorms  $|\bu(\varepsilon)|_{T,\Omega}^M$ and $|\overline{\bu(\varepsilon)}|_{T,\omega}^M$, the seminorms of the respective time derivatives   and the norms $\left\| \varepsilon \bu(\varepsilon)\right\|_{\WHMOt}$ and $|\eij(\varepsilon)|_{\WLO}$ are bounded independently of $\var$, $0<\varepsilon\leq\varepsilon_2$.
			
			Furthermore, by the definition of the spaces  $V_M^{\#}(\Omega)$ and $V_M^{\#}(\omega)$ in (\ref{espacio_Omega_1tipo})--(\ref{espacio_omega1tipo}), there exists a subsequence, also denoted by  $(\bu(\varepsilon))_{\varepsilon>0}$, and there exist  $\bu\in \WVMO $, $\bu^{-1}=(u_i^{-1})\in \WVO$, $\eij\in \WLO$ and $\bxi\in \WVMo$ such that
			\begin{align}\nonumber
			\bu(\varepsilon)&\deb \bu \en \WVMO, 
			\\\nonumber
			\varepsilon \bu(\varepsilon)&\deb \bu^{-1}\en \WHOt \ \textrm{hence,} \ \varepsilon \bu(\varepsilon)\rightarrow \bu^{-1} \en H^{1}(0,T;[L^2(\Omega)]^3),
			\\ \nonumber
			\eij(\varepsilon) &\deb \eij \en \WLO,
			\\ \nonumber
			\d_3u_3(\varepsilon)&=\varepsilon\edtres(\varepsilon) \rightarrow 0 \en \WLO,
			\\ \nonumber
			\overline{\bu(\varepsilon)} &\deb\bxi \en \WVMo ,
			\end{align}
			when $\varepsilon \rightarrow 0$.}
		
		Let $\bv=\bu(\varepsilon)$ in (\ref{ecuacion_fadmisibles}), then
		\begin{align} \nonumber
		&\int_{\Omega}A^{ijkl}(\varepsilon)e_{k||l}(\varepsilon)e_{i||j}(\varepsilon)\sqrt{g(\varepsilon)} dx
		+ \frac{1}{2}\frac{\d}{\d t}\int_{\Omega}B^{ijkl}(\varepsilon)\ekl(\varepsilon)\eij(\varepsilon)\sqrt{g(\varepsilon)} dx
		= L(\var)(\bu(\var)), 
		\end{align} 
		$\ae.$ Integrating over the interval $[0,T]$, using (\ref{elipticidadB_eps}) and (\ref{condicion_inicial_def}) we obtain that  
		\begin{align}\label{ref4}
		\int_0^T\left(\int_{\Omega}A^{ijkl}(\varepsilon)\ekl(\varepsilon)\eij(\varepsilon)\sqrt{g(\varepsilon)} dx \right)dt\leq \int_0^TL(\varepsilon)(\bu(\varepsilon))dt.
		\end{align}
		Now, by (\ref{fuerzasadmisibles}) and the Cauchy-Schwarz inequality we find that
		\begin{align} \nonumber
		\int_0^T L(\var)(\bu(\var)) dt &\leq \tilde{K}_0 \int_0^T \left( \sum_{i,j}|\eij(\varepsilon)|^2_{0,\Omega} \right)^{1/2} dt 
		\\ \label{ref1}
		&\leq \tilde{K}_0\sqrt{T} \left(\int_0^T \left( \sum_{i,j}|\eij(\varepsilon)|^2_{0,\Omega}\right) dt \right)^{1/2},
		\end{align}
		where $\tilde{K}_0:=\int_0^T K_0(t)dt > 0.$On the other hand, by (\ref{g_acotado}), (\ref{elipticidadA_eps}) and step $(i)$ we have that
		\begin{align} \nonumber
		c_0^{-2} C_e^{-1}g_0^{1/2}\left(|\bu(\varepsilon)|_{T,\Omega}^M\right)^2 &\leq C_e^{-1}g_0^{1/2} \int_0^T\left(\sum_{i,j}|\eij(\varepsilon)|^2_{0,\Omega}\right)dt
		\\ \label{ref3}
		\leq  \int_0^T&\left(\int_{\Omega}A^{ijkl}(\varepsilon)\ekl(\varepsilon)\eij(\varepsilon)\sqrt{g(\varepsilon)} dx\right) dt.
		\end{align}
		Now, (\ref{ref4})--(\ref{ref3})  together imply that  $|\eij(\var)|_{\LLO}$ is bounded and, as a consequence, $|\bu(\var)|^M_{T,\Omega}$ and $|\overline{\bu(\var)}|^M_{T,\omega}\leq|\bu(\var)|^M_{T,\Omega}$ do as well. By the Theorem \ref{Th_desigKorn} it follows that $||\var\bu(\var)||_{L^{2}(0,T;[H^1(\Omega)]^3)}$ is bounded. 
		
		Let $\bv=\dot{\bu}(\varepsilon)$ in (\ref{ecuacion_fadmisibles}), then
		\begin{align} \nonumber
		&\frac{1}{2}\frac{\d}{\d t}\int_{\Omega}A^{ijkl}(\varepsilon)e_{k||l}(\varepsilon)e_{i||j}(\varepsilon)\sqrt{g(\varepsilon)} dx
		\\& \quad + \int_{\Omega}B^{ijkl}(\varepsilon)\dekl(\varepsilon)\deij(\varepsilon)\sqrt{g(\varepsilon)} dx
		= L(\var)(\dot{\bu}(\var)), 
		\end{align} 
		$\ae.$ Integrating over  $[0,T]$, using (\ref{elipticidadA_eps}) and (\ref{condicion_inicial_def}) we obtain that  
		\begin{align}\label{ref5}
		\int_0^T\left(\int_{\Omega}B^{ijkl}(\varepsilon)\dekl(\varepsilon)\deij(\varepsilon)\sqrt{g(\varepsilon)} dx \right)dt\leq \int_0^TL(\varepsilon)(t)(\dot{\bu}(\varepsilon))dt,
		\end{align}
		that is analogous to (\ref{ref4}) with the contravariant components of the viscosity tensor instead. Hence, using similar arguments and (\ref{elipticidadB_eps}), we find that
		$\left|\deij(\var)\right|_{\LLO}$ are bounded and, as a consequence, $\left|\dot{\bu}(\var)\right|^M_{T,\Omega}$ and $\left|\dot{\overline{\bu(\var)}}\right|^M_{T,\omega}\leq\left|\dot{\bu}(\var)\right|^M_{T,\Omega}$ do as well. By the Theorem \ref{Th_desigKorn} it follows that $\left\| \var\dot{\bu}(\var)\right\| _{L^{2}(0,T;[H^1(\Omega)]^3)}$ is bounded. Therefore, the \textit{a priori} boundedness and convergences announced in this step are verified.

		\item { \em We obtain  expressions and relations for the limits $\eij$ found in the previous step.}
		
		Let $\bv=(v_i)\in V(\Omega)$. Then, by the definitions (\ref{eab})--(\ref{edtres}),
		\begin{align}\nonumber
		\varepsilon\eab(\varepsilon;\bv)&\rightarrow 0 \ \textrm{in} \ L^2(\Omega),
		\\\nonumber
		\varepsilon\eatres(\varepsilon;\bv)&\rightarrow \frac1{2}\d_3v_\alpha \  \textrm{in} \ L^2(\Omega),
		\\\nonumber
		\varepsilon\edtres(\varepsilon;\bv)&=\d_3v_3 \ \textrm{for all} \ \varepsilon>0.
		\end{align}
		
		
		Let $\bv=\varepsilon \bv \in V(\Omega)$ in (\ref{ecuacion_fadmisibles}) 
		and let $\varepsilon \rightarrow 0 $. As a consequence of the asymptotic behaviour of the functions $\varepsilon\eij(\varepsilon;\bv)$ above, the function $g(\varepsilon)$ and   the contravariant components of the fourth order tensors $A^{ijkl}(\varepsilon)$ and $B^{ijkl}(\varepsilon)$ (see Theorems \ref{Th_simbolos2D_3D} and \ref{Th_comportamiento asintotico}, respectively), the convergences of the admissible functions $F^{ij}(\var)$ and the weak convergences found in $(ii)$, we obtain that
		\begin{align}\nonumber
		&\int_{\Omega}2\mu a^{\alpha\sigma}\eatres\d_3v_\sigma + (\lambda + 2\mu)\edtres\d_3v_3\sqrt{a}dx
		+\int_{\Omega}\lambda a^{\alpha\beta}\eab\d_3v_3\sqrt{a}dx
		\\ \nonumber
		&\qquad  +\int_{\Omega}\rho a^{\alpha\sigma}\deatres\d_3v_\sigma+(\theta + \rho){\dedtres}\d_3v_3\sqrt{a} dx  +\int_{\Omega}\theta a^{\alpha\beta}{\deab}\d_3v_3\sqrt{a}dx
		\\  \label{ecuacion_integral_sus}
		&\quad= \int_{\Omega}\left( F^{\alpha 3} \d_3v_\alpha + F^{33}\d_3 v_3 \right)  \sqrt{a} dx, \ae.
		\end{align}
		Let $\bv\in V(\Omega)$ be independent of $x_3$. Then, we have that
		\begin{align}\nonumber
		&\int_{\Omega}2\mu a^{\alpha\sigma}\eatres\d_3v_\sigma\sqrt{a}dx
		+\int_{\Omega}\rho a^{\alpha\sigma}\deatres\d_3v_\sigma\sqrt{a} dx
		= \int_{\Omega} \left(F^{\alpha 3} \d_3v_\alpha \right) \sqrt{a} dx.
		\end{align}
		
		Hence, by Theorem \ref{th_int_nula} this equation leads to,
		\begin{align}\nonumber
		&2\mu a^{\alpha\sigma}\eatres
		+\rho a^{\alpha\sigma}\deatres
		=  F^{\sigma 3},
		\end{align}
		and using that $(a_{\alpha\sigma})^{-1}=\left(a^{\alpha \sigma}\right)$, we obtain the following ordinary differential equation, 
		\begin{align}\label{ecuacion_casuistica1}
		{2\mu}\eatres + \rho\deatres =a_{\alpha\sigma}F^{\sigma3}.
		\end{align}
		\begin{remark} \label{nota_desvio1}
			Note that removing time dependency and viscosity (taking $\rho=0$), the equation leads to the one studied in \cite{Ciarlet4b}, that is, the elastic case. 
		\end{remark}
		In order to solve the equation (\ref{ecuacion_casuistica1}) in the more general case, we assume that the viscosity coefficient $\rho$ is strictly positive. Moreover, we can prove that this equation is equivalent to 
		\begin{align} \nonumber
		& \frac{\d}{\d t}\left( e^{\frac{2\mu}{\rho} t} \eatres(t)\right) = \frac1{\rho} a_{\alpha \sigma} e^{\frac{2\mu}{\rho} t } F^{\sigma 3}(t).
		\end{align}
		Integrating with respect to the time variable and using (\ref{condicion_inicial_def}) we find that
		\begin{align}\label{eatres_longmemory}
		\eatres(t)=\frac1{\rho} a_{\alpha\sigma}\int_{0}^{t} e^{-\frac{2\mu}{\rho} (t-s)}F^{\sigma 3}(s)ds  \en \Omega, \forallt.
		\end{align}
		Moreover, from (\ref{ecuacion_casuistica1}) we obtain that,
		\begin{align*}
		\deatres(t)= \frac{1}{\rho}\left( a_{\alpha\sigma} F^{\sigma 3}(t) - 2\mu\eatres(t)\right) \en \Omega, \ae.
		\end{align*}
		Now, take in (\ref{ecuacion_integral_sus}) $\bv\in V(\Omega)$ such that $v_\alpha=0,$ then we have that
		\begin{align}\nonumber \label{ecua_int}
		&\int_{\Omega} (\lambda + 2\mu)\edtres\d_3v_3\sqrt{a}dx
		+\int_{\Omega}\lambda a^{\alpha\beta}\eab\d_3v_3\sqrt{a}dx +\int_{\Omega}(\theta + \rho){\dedtres}\d_3v_3\sqrt{a} dx
		\\
		&\qquad + \int_{\Omega}\theta a^{\alpha\beta}{\deab}\d_3v_3\sqrt{a}dx=
		\int_{\Omega}F^{33}\d_3v_3 \sqrt{a}dx.
		\end{align}
		Applying Theorem \ref{th_int_nula}, we  obtain the following differential equation,
		\begin{align} \label{ecuacion_casuistica}
		\lambda a^{\alpha\beta} \eab + (\lambda+2\mu) \edtres +\theta a^{\alpha \beta} \deab+ (\theta + \rho) \dedtres=F^{33}.
		\end{align}
		\begin{remark} \label{nota_desvio}
			Once again, note that removing time dependency and viscosity (taking $\theta=\rho=0$), the equation leads to the one studied in \cite{Ciarlet4b}, that is, the elastic case. 
		\end{remark}
		
		In order to solve the equation (\ref{ecuacion_casuistica}) in the more general case, we assume that the viscosity coefficient $\theta$ is strictly positive.  Moreover, we can prove that this equation is equivalent to
		\begin{align} \label{hipotesis}
		\theta e^{-\frac{\lambda}{\theta}t} \frac{\d}{\d t}\left(a^{\alpha\beta}\eab (t) e^{\frac{\lambda}{\theta}t}\right)= F^{33}(t)-\left(\theta + \rho \right)e^{-\frac{\lambda + 2\mu}{\theta + \rho}t} \frac{\d}{\d t}\left(\edtres (t) e^{\frac{\lambda + 2\mu}{\theta + \rho}t}\right).
		\end{align}
		Integrating respect to the time variable, using (\ref{condicion_inicial_def}) and simplifying  we find,
		\begin{align*}
		\edtres(t)= \frac{1}{\theta + \rho} \int_0^t e^{-k(t-s)} F^{33}(s)ds - \frac{\theta}{\theta + \rho} \left( a^{\alpha \beta }\eab(t) + \Lambda\int_0^te^{-k(t-s)}a^{\alpha\beta}\eab(s) ds \right),
		\end{align*}
		in $\Omega$ , $\forallt$, and where $\Lambda$ and $k$ are defined  in (\ref{Lambda}) and (\ref{k}), respectively. Moreover, from (\ref{ecuacion_casuistica}) we obtain that,
		\begin{align*}
		\dedtres(t)= \frac{1}{\theta + \rho}F^{33}(t)- \frac{\lambda}{\theta + \rho} a^{\alpha \beta} \eab(t)- \frac{\lambda + 2\mu}{\theta + \rho} \edtres (t) - \frac{\theta}{\theta + \rho} a^{\alpha\beta}\deab(t),
		\end{align*}
		in $\Omega$ , $\ae.$

		\item { \em The family $(\bu(\var)  )_{\var>0}  $ verifies
			\begin{equation}
			\left( \overline{\eab(\var)}  -\gab(\overline{\bu(\var)}) \right) \rightarrow 0 \ \textrm{in} \ \WLo \ \textrm{when} \ \var\rightarrow 0.
			\end{equation}
			As a consequence, the subsequence considered in $(ii)$ verifies
			\begin{equation}
			\gab(\overline{\bu(\var)}) \deb \overline{\eab} \ \textrm{in} \ \WLo.
			\end{equation}}
		This proof is a corollary of the step $(iv)$ in Theorem 5.6-1, \cite{Ciarlet4b}. We follow the same arguments made there but using Theorem \ref{Th_medias} $(a)$ and $(b)$. Then, the conclusion follows.

		\item {\em We obtain the equations satisfied by the limits $\eab$ found in the step $(ii).$}

		Let $\bv=(v_i)\in V(\Omega)$ be independent of the transversal variable $x_3$. Then, by the definitions (\ref{eab})--(\ref{edtres}),
		\begin{align}\nonumber
		&\eab(\varepsilon;\bv)\rightarrow \gab(\bv) \ \textrm{in} \ L^2(\Omega),
		\\\nonumber
		&\eatres(\varepsilon;\bv)\rightarrow \frac1{2}\d_\alpha v_3+ b_\alpha^\sigma v_\sigma\  \textrm{in} \ L^2(\Omega),
		\\\nonumber
		&\edtres(\varepsilon;\bv)=0 \ \textrm{for all} \ \varepsilon>0.
		\end{align}
		Keep such a function $\bv\in V(\Omega)$ in (\ref{ecuacion_fadmisibles}) and take the limit when $\var \to 0$. In the right-hand side of that equation, we have that
		\begin{equation} \label{limite_L_1}
		\lim_{\var\rightarrow0}L(\var)(\bv)=\int_{\Omega}\left( F^{\alpha\beta}\gab(\bv) + 2F^{\alpha 3} \left(\frac1{2}\d_\alpha v_3+b_\alpha^\sigma v_\sigma\right)\right)\sqrt{a}dx.
		\end{equation}
		In the left-hand side of the equation, by the asymptotic behaviour of functions  $g(\var)$ and the contravariant components of the fourth order tensors $A^{ijkl}(\var)$ and $B^{ijkl}(\var)$   (see Theorem \ref{Th_simbolos2D_3D}  and \ref{Th_comportamiento asintotico}, respectively), the convergences of the strain tensor components $\eij(\var;\bv)$  above and the weak convergences of $\eij(\var)\deb\eij$ in $\WLO$ found in step $(ii)$, we observe that,
		\begin{align}\nonumber
		& \int_{\Omega} A^{ijkl}(0)\ekl(\var)\eij(\var;\bv)\sqrt{a}dx +\int_{\Omega} B^{ijkl}(0)\dekl(\var)\eij(\var;\bv)\sqrt{a}dx
		\\ \nonumber
		&=\int_{\Omega}\left(\lambda a^{\alpha\beta}a^{\sigma\tau} + \mu\ten\right)\est\gab(\bv)\sqrt{a}dx 
		\\\nonumber
		& \quad +\int_{\Omega} \lambda a^{\alpha\beta} \edtres\gab(\bv)\sqrt{a}dx
		+ \int_{\Omega} 4\mu a^{\alpha\sigma}\estres \left(\frac1{2}\d_\alpha v_3+b_\alpha^\tau v_\tau\right)\sqrt{a}dx
		\\ \nonumber
		& \quad + \int_{\Omega} \left(\theta a^{\alpha\beta}a^{\sigma\tau}+ \frac{\rho}{2}\ten\right)\dest\gab(\bv)\sqrt{a}dx
		\\
		& \ \  +\int_{\Omega} \theta a^{\alpha\beta}\dedtres\gab(\bv)\sqrt{a}dx   
		+ \int_{\Omega} 2\rho a^{\alpha\sigma}\destres \left(\frac1{2}\d_\alpha v_3+b_\alpha^\tau v_\tau\right)\sqrt{a}dx,
		\end{align}
		which, using  the relations found in $(iii)$ and simplifying yields that,
		\begin{align} \nonumber
		& \frac{1}{2}\intO \a \est\gab({\bv})\sqrt{a}dx + \frac{1}{2}\intO \b \dest\gab({\bv})\sqrt{a}dx
		\\ \nonumber
		& \qquad- \frac{1}{2}\int_0^te^{-k(t-s)}\intO \c\est(s)\gab({\bv})\sqrt{a}dx ds 
		\\ \nonumber
		&\qquad + \intO \frac{\theta\Lambda}{\theta + \rho} \int_0^t e^{-k(t-s)} F^{33}(s)ds a^{\alpha\beta}\gab(\bv)\sqrt{a}dx
		\\ \nonumber
		&\qquad + \intO \frac{\theta}{\theta + \rho} F^{33} a^{\alpha\beta}\gab(\bv)\sqrt{a} dx 
		+\intO 2F^{\alpha 3} \left( \frac{1}{2} \d_\alpha v_3 + b_\alpha^\sigma v_\sigma\right) \sqrt{a} dx,
		\end{align}
		where  $\a$ , $\b$ and $\c$ denote the contravariant components of the two-dimensional  fourth order tensors defined  in (\ref{tensor_a_bidimensional})--(\ref{tensor_c_bidimensional}). Hence, together with (\ref{limite_L_1}) leads to
		\begin{align} \nonumber
		&\into \a \mest\gab(\bar{\bv})\sqrt{a}dy + \into \b \dmest\gab(\bar{\bv})\sqrt{a}dy 
		\\ \nonumber
		& \qquad- \int_0^te^{-k(t-s)}\into \c\mests\gab(\bar{\bv})\sqrt{a}dy ds 
		\\ \nonumber
		& \quad =\int_{\omega} \int_{-1}^{1} F^{\alpha\beta}dx_3\gab(\bar{\bv})\sqrt{a}dy  
		- \into \int_{-1}^{1}\frac{\theta\Lambda}{\theta + \rho} \int_0^t e^{-k(t-s)} F^{33}(s)ds a^{\alpha\beta}dx_3\gab(\bar{\bv})\sqrt{a}dy
		\\ \label{et2}
		&\qquad -\into \int_{-1}^{1} \frac{\theta}{\theta + \rho} F^{33} a^{\alpha\beta}dx_3\gab(\bar{\bv})\sqrt{a} dy= \into \varphi^{\alpha\beta}\gab(\bar{\bv})\sqrt{a}dy,
		\end{align}
		where  $\varphi^{\alpha\beta}$ denotes the real function defined in (\ref{phi_1}). Now, given $\beeta\in V(\omega)$, there exists a function $\bv\in V(\Omega)$ independent of $x_3$ such that $\overline{\bv}=\beeta$. Hence  (\ref{et2}) holds for all $\beeta\in V(\omega), \ae.$
		
		\item { \em The subsequence $(\bu(\var))_{\var>0}$ from $(ii)$ satisfies
			\begin{align}\label{stepvi_desig1}
			\var \bu(\var) &\deb \bcero \ \textrm{in} \ H^{1}(0,T;[H^1(\Omega)]^3), 
			\\ \label{stepvi_desig2}
			\d_3 u_\alpha(\var) &\deb 0 \ \textrm{in} \ \WLO, 
			\end{align}
			when $\var\rightarrow0.$ Moreover, $\eab$ are independent of the transversal variable $x_3$.}
		
		
		By the step $(ii)$ the functions $\bu^{-1}(\var):=\var\bu(\var)\in \WVO$  satisfy 
		\begin{align}\nonumber
		&\bu^{-1}(\var)\deb\bu^{-1} \en \WHOt,
		\\\label{ref16}
		& \textrm{hence} \ \bu^{-1}(\var)\to\bu^{-1} \en H^{1}(0,T;[L^2(\Omega)]^3),
		\\\label{ref17}
		&\frac{1}{\var}\eij(\var;\bu^{-1}(\var))\deb\eij \en \WLO.
		\end{align}
		Hence, by Theorem \ref{Th_522}, $\overline{\bu^{-1}}\in V_F(\omega)$ and consequently $\overline{\bu^{-1}}=\bcero$, since $V_F(\omega)=\{\bcero\}$ by assumption. By the same result, $\bu^{-1}$ is independent of $x_3$, hence $\bu^{-1}=\bcero$ and (\ref{stepvi_desig1}) is proved. Moreover, this implies that $\var\bu(\var)\to\bcero$ in $H^{1}(0,T;[L^2(\Omega)]^3)$. Now, by (\ref{eatres}) we have that
		\begin{align*}
		\d_3u_\alpha(\var)= 2\var\eatres(\var)-\var\d_\alpha u_3(\var) + 2\var\Gamma_{\alpha 3}^\sigma(\var)u_\sigma(\var).
		\end{align*}
		Therefore, together with the convergences in $(ii)$ and above and the boundedness of the sequence $(\Gamma_{\alpha3}^\sigma(\var))_{\var>0}$ in $\mathcal{C}^0(\bar{\Omega})$ by the Theorem \ref{simbolos3D}, imply that $\d_3 u_\alpha(\var)\deb 0$ in $\WLO$ and (\ref{stepvi_desig2}) is proved. Moreover, since  $\bu^{-1}=\bcero$ and (\ref{ref16})--(\ref{ref17}), taking $\bu=\bu^{-1}$ in Theorem \ref{Th_522} we have that $\d_3\eab=-\rab(\bu^{-1})=0.$ Therefore, the functions $\eab$ are independent of $x_3$.
		
		\item {\em The following strong convergences are satisfied,
			\begin{align*}
			\eij(\var)&\rightarrow\eij \en \WLO ,
			\\
			\var\bu(\var)&\rightarrow \bcero \en \WHOt,
			\\
			\gab(\overline{\bu(\var)}) &\rightarrow \overline{\eab} \en \WLo,
			\\
			\overline{\bu(\var)}&\rightarrow \xi \en \WVMo,
			\end{align*}
			$\Forallt.$}
		
		Let us define,
		\begin{align}\nonumber
		\Psi(\varepsilon)&:=\int_{\Omega}A^{ijkl}(\varepsilon)(\ekl(\varepsilon)-\ekl)(\eij(\varepsilon)-\eij)\sqrt{g(\varepsilon)}dx \\ \nonumber
		& \quad +\int_{\Omega}B^{ijkl}(\varepsilon)(\dekl(\varepsilon)-\dekl)(\eij(\varepsilon)-\eij)\sqrt{g(\varepsilon)}dx
		\\ \nonumber
		&=L(\var)(\bu(\var)) - \int_{\Omega} A^{ijkl}(\varepsilon)(2\ekl(\varepsilon)-\ekl)\eij\sqrt{g(\varepsilon)}dx
		\\ \nonumber
		& \quad + \int_{\Omega} B^{ijkl}(\varepsilon)(\dekl\eij - \frac{\d}{\d t}(\ekl(\var)\eij))\sqrt{g(\varepsilon)}dx, \ae.
		\end{align}
		We have that,
		\begin{align}\nonumber
		&\int_{\Omega}A^{ijkl}(\varepsilon)(\ekl(\varepsilon)-\ekl)(\eij(\varepsilon)-\eij)\sqrt{g(\varepsilon)}dx \\ \nonumber
		& \quad +\frac{1}{2}\frac{\d}{\d t}\int_{\Omega}B^{ijkl}(\varepsilon)(\ekl(\varepsilon)-\ekl)(\eij(\varepsilon)-\eij)\sqrt{g(\varepsilon)}dx
		=\Psi(\varepsilon), \ae.
		\end{align}
		Integrating over the interval $[0,T]$, using (\ref{elipticidadB_eps}) and (\ref{condicion_inicial_def}) we find that
		\begin{align}\label{Landa_eps}
		\int_0^T\left(\int_{\Omega}A^{ijkl}(\varepsilon)(\ekl(\varepsilon)-\ekl)(\eij(\varepsilon)-\eij)\sqrt{g(\varepsilon)}dx \right) dt
		\leq \int_0^T \Psi(\varepsilon) dt.
		\end{align}
		Now, by (\ref{g_acotado}) and (\ref{elipticidadA_eps}) we have
		\begin{align}\nonumber
		{C}_e^{-1}g_0^{1/2}\sum_{i,j}|\eij(\varepsilon)-\eij|^2_{0,\Omega}
		\leq\int_{\Omega}A^{ijkl}(\varepsilon)(\ekl(\varepsilon)-\ekl)(\eij(\varepsilon)-\eij)\sqrt{g(\varepsilon)}dx.
		\end{align}
		Therefore, together with the previous inequality leads to
		\begin{align}\nonumber
		{C}_e^{-1}g_0^{1/2}\int_0^T \left(\sum_{i,j}|\eij(\varepsilon)-\eij|^2_{0,\Omega}\right) dt
		\leq  \int_0^T \Psi(\varepsilon) dt.
		\end{align}
		On the other hand, the strong convergences $F^{ij}(\var)\rightarrow F^{ij}$ in $\LLO$ given by assumption and the weak convergences $\eij(\var)\deb \eij \en \WLO$ from $(ii)$, imply that
		\begin{align}\nonumber
		\lim_{\var\rightarrow 0} &L(\var)(\bu(\var))
		= \lim_{\var\rightarrow 0} \left(\int_{\Omega} F^{\alpha\beta}(\var)\eab(\var)\sqrt{g(\var)} dx\right.
		\\ \nonumber
		& \left. \qquad + \int_{\Omega} (2F^{\alpha 3}(\var) \eatres(\var) + F^{33}(\var)\edtres )\sqrt{g(\var)}dx \right)
		\\\label{et4}
		& =\intO F^{\alpha\beta}\eab\sqrt{a} dx + \intO \left(2F^{\alpha 3}\eatres + F^{33}\edtres  \right)\sqrt{a}dx.
		\end{align}
		Now, by the asymptotic behaviour of functions  $g(\var)$ and the contravariant components of the fourth order tensors $A^{ijkl}(\var)$ and $B^{ijkl}(\var)$   (see Theorem \ref{Th_simbolos2D_3D}  and \ref{Th_comportamiento asintotico}, respectively), we find that
		\begin{align}\nonumber
		&\lim_{\var\rightarrow 0}\left(\int_{\Omega} A^{ijkl}(\varepsilon)(2\ekl(\varepsilon)-\ekl)\eij\sqrt{g(\varepsilon)}dx \right.
		\\  \nonumber
		& \left.\qquad-\int_{\Omega} B^{ijkl}(\varepsilon)(\dekl\eij - \frac{\d}{\d t}(\ekl(\var)\eij))\sqrt{g(\varepsilon)}dx\right)
		\\\nonumber
		& \quad= \int_{\Omega} A^{ijkl}(0)\ekl\eij\sqrt{a}dx + \int_{\Omega} B^{ijkl}(0)\dekl\eij\sqrt{a}dx
		\\ \nonumber
		&\quad=\int_{\Omega}\left(\lambda a^{\alpha\beta}a^{\sigma\tau} + \mu\ten\right)\est\eab\sqrt{a}dx  +\int_{\Omega} \lambda a^{\alpha\beta} \edtres\eab\sqrt{a}dx
		\\\nonumber
		& \quad
		+\int_{\Omega} 4\mu a^{\alpha\sigma}\estres\eatres\sqrt{a}dx
		+ \intO \left( \lambda a^{\sigma \tau} \est + \left( \lambda + 2 \mu \right) \edtres  \right)\edtres \sqrt{a}   dx
		\\ \nonumber
		& \quad + \int_{\Omega} \left(\theta a^{\alpha\beta}a^{\sigma\tau}+ \frac{\rho}{2}\ten\right)\dest\eab\sqrt{a}dx  +\int_{\Omega} \theta a^{\alpha\beta}\dedtres\eab\sqrt{a}dx   
		\\ \nonumber
		& \quad
		+ \int_{\Omega} 2\rho a^{\alpha\sigma}\destres \eatres\sqrt{a}dx
		+ \intO \left( \theta a^{\sigma \tau} \dest + \left( \theta + \rho\right) \dedtres  \right)\edtres \sqrt{a}   dx,
		\end{align}
		which substituting the findings in the step $(iii)$ and simplifying leads to
		\begin{align} \nonumber
		&\into \a \mest \ \overline{\eab}\sqrt{a}dy + \into \b \dmest \ \overline{\eab}\sqrt{a}dy 
		\\ \nonumber
		&\qquad - \int_0^te^{-k(t-s)}\into \c\mests \ \overline{\eab}\sqrt{a}dy ds 
		\\ \nonumber
		& \qquad 
		+ \into  \int_{-1}^{1}  \left( \frac{\theta}{\theta + \rho} F^{33} a^{\alpha\beta} + \frac{\theta\Lambda}{\theta + \rho} \int_0^t e^{-k(t-s)} F^{33}(s)ds a^{\alpha\beta}  \right)dx_3      \overline{\eab}\sqrt{a}dy
		\\ \label{et3}
		& \qquad +\intO\left( 2 F^{\alpha 3} \eatres +  F^{33} \edtres\right) \sqrt{a} dx, 
		\end{align}
		where  $\a$ , $\b$ and $\c$ denote the contravariant components of the  fourth order tensors defined in (\ref{tensor_a_bidimensional})--(\ref{tensor_c_bidimensional}). Hence, together with (\ref{et4}), we have that
		\begin{align} \nonumber
		\Psi:=\lim_{\var\rightarrow 0} \Psi(\var)&= \into \varphi^{\alpha \beta} \overline{\eab} \sqrt{a} dy -\into \a \mest \ \overline{\eab}\sqrt{a}dy - \into \b \dmest \ \overline{\eab}\sqrt{a}dy 
		\\ \label{ref10}
		& \qquad
		+\int_0^te^{-k(t-s)}\into \c\mests \ \overline{\eab}\sqrt{a}dy ds ,
		\end{align}
		with $\varphi^{\alpha\beta}$ defined in (\ref{phi_1}). Now, since $\overline{\bu(\var)} \in V(\omega)$, for each $\var>0$, we take $\bar{\bv}=\beeta= \overline{\bu(\var)} $     in (\ref{et2}) and we have that
		\begin{align} \nonumber
		&\into \a \mest\gab(\overline{\bu(\var)})\sqrt{a}dy + \into \b \dmest\gab(\overline{\bu(\var)})\sqrt{a}dy 
		\\ \label{ref9}
		& \qquad- \int_0^te^{-k(t-s)}\into \c\mests\gab(\overline{\bu(\var)})\sqrt{a}dy ds 
		= \into \varphi^{\alpha\beta}\gab(\overline{\bu(\var)})\sqrt{a}dy  .
		\end{align}
		Taking in (\ref{ref9}) the limit when $\var\rightarrow 0$ together with the weak convergences in $(iv)$, we conclude from (\ref{ref10}) that $\Psi=0$ . As a consequence,  using  the Lebesgue dominated convergence theorem in (\ref{Landa_eps}),  the strong convergences $\eij(\var)\rightarrow \eij \en \LLO$ are verified.
		Analogously, if we define
		\begin{align}\nonumber
		\tilde{\Psi}(\varepsilon)&:=\int_{\Omega}A^{ijkl}(\varepsilon)(\ekl(\varepsilon)-\ekl)(\deij(\varepsilon)-\deij)\sqrt{g(\varepsilon)}dx \\ \nonumber
		& \qquad+\int_{\Omega}B^{ijkl}(\varepsilon)(\dekl(\varepsilon)-\dekl)(\deij(\varepsilon)-\deij)\sqrt{g(\varepsilon)}dx
		\\ \nonumber
		&\quad= L(\var)(\dot{\bu}(\var)) +\int_{\Omega} A^{ijkl}(\varepsilon)(\ekl\deij - \frac{\d}{\d t}(\ekl(\var)\eij))\sqrt{g(\varepsilon)}dx
		\\ \nonumber
		&\qquad -\int_{\Omega} B^{ijkl}(\varepsilon)(2\dekl(\varepsilon)-\dekl)\deij\sqrt{g(\varepsilon)} dx.
		\end{align} 
		We have that,
		\begin{align}\nonumber
		\frac{1}{2}&\frac{\d}{\d t}\int_{\Omega}A^{ijkl}(\varepsilon)(\ekl(\varepsilon)-\ekl)(\eij(\varepsilon)-\eij)\sqrt{g(\varepsilon)}dx
		\\ \nonumber
		& +\int_{\Omega}B^{ijkl}(\varepsilon)(\dekl(\varepsilon)-\dekl)(\deij(\varepsilon)-\deij)\sqrt{g(\varepsilon)}dx =\tilde{\Psi}(\varepsilon), \ae.
		\end{align}
		Integrating over  $[0,T]$, using (\ref{elipticidadA_eps}) and (\ref{condicion_inicial_def}) we find that
		\begin{align}\nonumber
		\int_0^T\left(\int_{\Omega}B^{ijkl}(\varepsilon)(\dekl(\varepsilon)-\dekl)(\deij(\varepsilon)-\deij)\sqrt{g(\varepsilon)}dx\right) dt
		\quad\leq \int_{0}^{T}\tilde{\Psi}(\varepsilon)dt,
		\end{align}
		Now, by (\ref{elipticidadB_eps}) and (\ref{g_acotado})
		\begin{align}\nonumber
		&{C}_v^{-1}g_0^{1/2}\sum_{i,j}|\deij(\varepsilon)-\deij|^2_{0,\Omega}
		\leq\int_{\Omega}B^{ijkl}(\varepsilon)(\dekl(\varepsilon)-\dekl)(\deij(\varepsilon)-\deij)\sqrt{g(\varepsilon)}dx
		\end{align}
		Therefore, together with the previous inequality leads to
		\begin{align}\label{Landat_eps}
		&{C}_v^{-1}g_0^{1/2}\int_0^T\left(\sum_{i,j}|\deij(\varepsilon)(t)-\deij(t)|^2_{0,\Omega}\right)dt \leq \int_0^T \tilde{\Psi}(\varepsilon) dt,
		\end{align}
		which is similar with (\ref{Landa_eps}). Therefore, using analogous arguments as before, we find that
		\begin{align} \nonumber
		\tilde{\Psi}&:=\lim_{\var\rightarrow 0} \tilde{\Psi}(\var)=\into \varphi^{\alpha \beta} \dot{\overline{\eab}} \sqrt{a} dy -\into \a \mest \ \dmeab\sqrt{a}dy- \into \b \dmest \ \dmeab\sqrt{a}dy 
		\\ \label{ref11}
		& \qquad+ \int_0^te^{-k(t-s)}\into \c\mests \ \dmeab\sqrt{a}dy ds, \ae.
		\end{align}
		Now, since $\dot{\overline{\bu(\var)}} \in V(\omega)$, for each $\var>0$, we take $\bar{\bv}=\beeta= \dot{\overline{\bu(\var)}} $     in (\ref{et2}) and we have that
		\begin{align} \nonumber
		&\into \a \mest\gab(\dot{\overline{\bu(\var)}})\sqrt{a}dy + \into \b \dmest\gab(\dot{\overline{\bu(\var)}})\sqrt{a}dy 
		\\ \nonumber
		& \qquad- \int_0^te^{-k(t-s)}\into \c\mests\gab(\dot{\overline{\bu(\var)}})\sqrt{a}dy ds 
		\\\label{ref12}
		& \quad= \into \varphi^{\alpha\beta}\gab(\dot{\overline{\bu(\var)}})\sqrt{a}dy  .
		\end{align}
		Taking in (\ref{ref12}) the limit when $\var\rightarrow 0$ together with the weak convergences in $(iv)$, we conclude from (\ref{ref11}) that $\tilde{\Psi}=0$.   As a consequence, using the Lebesgue dominated convergence theorem in (\ref{Landat_eps}), the  strong convergences  $\deij(\varepsilon) \rightarrow\deij $ in $L^2(0,T;L^2(\Omega))$  are satisfied. Therefore, we conclude that $\eij(\varepsilon) \rightarrow\eij $ in $\WLO.$
		
		Now, let $\bv=\bu^{-1}(\var)=\var\bu(\var)$ in the second inequality in  Theorem \ref{Th_521}. We find by the step $(ii)$ that there exist two constants $\tilde{C},\hat{C}>0$ such that
		\begin{align*}
		&||\d_3\eab(\var)+ \rab(\bu^{-1}(\var))||_{\WHMO}\leq \tilde{C}\Big(\var\sum_i|\eitres(\var)|_{\WLO} 
		\\
		& \qquad + \var\sum_\alpha|\var u_\alpha(\var)|_{\WLO} 
		+ \var||\var u_3(\var)||_{\WHO}  \Big) \leq \hat{C}\var .
		\end{align*}
		Moreover, since $\eab(\var)\to \eab$ in $\WLO$ and the functions $\eab$ are independent of $x_3$ (see step $(iv)$) we have that
		\begin{align*}
		\d_3 \eab (\var) \to \d_3\eab =0 \en \WHMO,
		\end{align*}
		hence, from the previous inequality,
		\begin{align*}
		\rab(\bu^{-1}(\var)) \to 0 \en \WHMO.
		\end{align*}
		Now, applying Theorem \ref{Th_522} we have that
		\begin{align*}
		\bu^{-1}(\var)=\var\bu(\var) \to \bcero \en \WHOt.
		\end{align*}
		By the Theorem \ref{Th_medias}  (a), the strong convergences $\eij(\var)\to\eij$ in $\WLO$ imply that $\overline{\eab(\var)}\to \overline{\eab}$ in $\WLo$. Therefore, by $(iv)$ 
		\begin{align*}
		\gab(\overline{\bu(\var)})\to \overline{\eab} \en \WLo.
		\end{align*}
		As a consequence, $(\gab(\overline{\bu(\var)}))_{\var>0}$ is a Cauchy sequence in $\WLo$. Now, since 
		\begin{align*}
		\left|\overline{\bu(\var)}- \overline{\bu(\var')}  \right|^M_{T,\omega} = \left(\sum_{\alpha,\beta} \int_0^T \left|\gab(\overline{\bu(\var)(t)}) - \gab(\overline{\bu(\var')(t)}) \right|^M_{\omega} dt \right)^{1/2},
		\end{align*}
		with $\var,\var'>0$  and the corresponding identity for the time derivatives hold, the strong convergence $\overline{\bu(\var)} \to \bxi$ in $\WVMo$ is verified.

		
		\item {\em The limit $\bxi(t)\in V_M^{\#}(\omega) \forallt$ found in $(vii)$ satisfies the system of equations
			\begin{align*}
			&B_{M}^{\#}(\bxi,\beeta)=L_{M}^{\#}(\beeta) \ \forall\beeta\in V_M^{\#}(\omega), \aes, 
			\\
			&\bxi(0, \cdot)=\bxi_0(\cdot),
			\end{align*}
			which has a unique solution. Then, the convergence $\overline{\bu(\var)}\rightarrow\bxi \en \WVMo $ is verified by the all family $(\overline{\bu(\var)})_{\var>0}$.}
		
		Let $\beeta\in V(\omega)$. By the steps  $(v)$ and $(vii)$ and since $\overline{\bu(\var)}\in \WVo$, we find that,
		\begin{align*}
		&\lim_{\var\to 0} B_{M}(\overline{\bu(\var)}, \beeta)= \lim_{\var\to 0} \left( \into \a \gst(\overline{\bu(\var)}) \gab(\beeta) \sqrt{a}dy \right.
		\\
		& \left. \quad + \into \b \gst(\dot{\overline{\bu(\var)}}) \gab(\beeta) \sqrt{a}dy  - \int_0^t e^{-k(t-s)}\into \c \gst(\overline{\bu(\var)(s)}) \gab(\beeta) \sqrt{a} dy ds \right)
		\\
		&=\into \a\mest \gab(\beeta) \sqrt{a}dy + \into \b \dmest \gab(\beeta) \sqrt{a}dy 
		\\
		&  \quad - \int_0^t e^{-k(t-s)}\into \c \mests \gab(\beeta) \sqrt{a} dy ds=L_{M}(\beeta).
		\end{align*}
		Furthermore, again by $(vii)$ we have that, 
		\begin{align*}
		\lim_{\var\to 0} B_{M}(\overline{\bu(\var)}, \beeta)= B_{M}^{\#}(\bxi,\beeta)=L_{M}(\beeta), \forall \beeta\in V(\omega), \aes,
		\end{align*}
		hence, $B_{M}^{\#}(\bxi,\beeta)=L_{M}^{\#}(\beeta) \ \forall \beeta\in V_M^{\#}(\omega), \aes,$ by the definition of the continuous extensions $B_{M}^{\#}$ and $L_{M}^{\#}$. Besides, this problem has a unique  solution by Theorem 6.4,  \cite{intro2}.
		
		\item { \em Let $\Omega$ be a domain in $\mathbb{R}^3$. Given  $\bv=(v_i)\in H^{1}(0,T;[L^2(\Omega)]^3)$, we define the distributions,
			\begin{equation}\label{ref15}
			e_{ij}(\bv):= \frac1{2}(\d_iv_j + \d_jv_i)\in \WHMO.
			\end{equation}
			
			Considering a sequence of functions $\bv^k=(v_i^k)\in H^{1}(0,T;[L^2(\Omega)]^3)$ such that $\bv^k\rightarrow\bcero \en \WHMOt$ and $e_{ij}(\bv^k)\rightarrow0 \en \WHMO$ when $k\rightarrow\infty$. Then, $\bv^k\rightarrow \bcero \en H^{1}(0,T;[L^2(\Omega)]^3)$.}
		
		This proof is a generalization of the step $(ix)$ in Theorem 5.6-1, \cite{Ciarlet4b}. We follow
		the same arguments made there to prove that $\bv^k\rightarrow \bcero \en L^2(0,T;[L^2(\Omega)]^3)$  and the corresponding convergences of the time derivatives in the same
		space. Then the conclusion follows.

		\item { \em The following convergences are satisfied:
			\begin{align*}
			\bu(\var)&\rightarrow \bu \en \WVMO,
			\\\d_3u_\alpha(\var)&\rightarrow 0 \en \WLO.
			\end{align*}}
		
		In order to prove the first convergence it is enough to prove that  $(\bu(\var))_{\var>0}$ and its time derivative are Cauchy sequences with respect to the norm  $\left|\cdot\right|_{T,\Omega}^M$. By its definition we have,
		\begin{align}\nonumber
		\int_0^T \left( | \bu(\var)(t)- \bu(\var')(t)|_{\Omega}^M  \right)^2dt &= \int_0^T \left( \sum_{\alpha, \beta}|\gab(\overline{\bu(\var)(t)}) - \gab(\overline{\bu(\var')(t)})|_{0,\omega}^2 \right) dt\\\label{ref18}
		& \quad +\int_0^T \left( \sum_i|\d_3 u_i(\var)(t) - \d_3 u_i(\var')(t)|_{0,\Omega}^2   \right) dt
		\end{align}
		and the analogous equality for the time derivative family. 
		Then, let us start proving that   $\d_3 u_\alpha (\var) \to 0$ en $\WLO$, the second convergence announced. This convergence is fulfilled if   $\d_3 \bu' (\var) \to 0$ in $\WLO$, with
		\begin{align} \label{ref14}
		\bu'(\var)=(u_1(\var), u_2(\var),0).
		\end{align}
		
		By  step  $(ix)$, proving this is equivalent to prove the following convergences:
		\begin{align}\label{conv1}
		\d_3 \bu'(\var)&\to \bcero \en \WHMOt,
		\\\label{conv2}
		e_{ij}(\d_3\bu'(\var)) &\to 0 \en \WHMO.
		\end{align}
		
		By (\ref{eatres}), we can obtain that
		\begin{align}\label{ref13}
		\d_3 u_\alpha (\var) = 2 \var \eatres (\var) - \var \d_\alpha u_3(\var) + 2\var \Gamma^{\sigma}_{\alpha 3} (\var) u_\sigma(\var). 
		\end{align}
		
		Hence, since the sequence $(\Gamma_{\alpha 3}^\sigma(\var))_{\var>0}$ is bounded in  $\mathcal{C}^0(\bar{\Omega})$ (see Theorem  \ref{Th_simbolos2D_3D}) and by the convergences  $\var \eatres(\var)\to 0$, $\var u_i(\var)\to 0$ in $\WLO$ by steps $(vi)$ y $(vii)$, respectively, together imply that
		\begin{align} \label{ref8}
		\d_3u_\alpha(\var)\to 0 \en \WHMO.
		\end{align} 
		
		That is, (\ref{conv1}) is verified. In order to prove (\ref{conv2}), firstly, we have that $e_{33}(\d_3 \bu'(\var))=\d_3 e_{33}(\bu'(\var))=0$ by (\ref{ref14}). Now, the asymptotic behaviour of the functions  $\Gamma_{\alpha3}^\sigma(\var)$ (see Theorem \ref{Th_simbolos2D_3D}) and the convergences  $\eij(\var)\to\eij$ in $\WLO$ (see step $(vii)$) imply that  (see (\ref{eatres})),
		\begin{align*}
		\left(\d_3 u_\alpha(\var) + \var\d_\alpha u_3(\var) + 2 \var b_\alpha^{\sigma}u_\sigma(\var) \right)\to 0 \en \WLO,
		\end{align*}
		thus,
		\begin{align*}
		\left(\d_{33} u_\alpha(\var) + \var\d_{\alpha3} u_3(\var) + 2 \var b_\alpha^{\sigma}\d_3u_\sigma(\var) \right)\to 0 \en \WHMO.
		\end{align*}
		
		Since $\d_3 u_3(\var)\to 0$ in $\WLO$ and $\var u_\sigma(\var)\to0$ in $\WLO$ (see steps $(ii)$ y $(vi)$, respectively) we have that
		\begin{align*}
		2 e_{\alpha 3}(\d_3 \bu'(\var))= \d_{33}u_\alpha (\var) \to 0 \en \WHMO.
		\end{align*}
		
		Now, by step  $(vi)$  we have that $\eab(\var)\to \eab$ in $\WLO$ and $\d_3\eab=0$, then we infer that $\d_3\eab(\var)\to 0$ in $\WHMO$. Hence,
		\begin{align*}
		\d_3\eab(\var)=\left( \d_3 e_{\alpha \beta} (\bu(\var)) - \d_3(\Gamma^p_{\alpha\beta}(\var)u_p(\var)) \right)\to 0 \en \WHMO.
		\end{align*}
		
		Since $\Gamma_{\alpha \beta}^p (\var)\in \mathcal{C}^1(\bar{\Omega})$ (by its definition, see  (\ref{simbolos3D}) and (\ref{escalado_simbolos}) ), then $\Gamma_{\alpha \beta}^p (\var)u_p(\var)\in\WHO$. Moreover, 
		\begin{align*}
		\d_3 ( \Gamma_{\alpha \beta}^p (\var)u_p(\var))&= \d_3\Gamma_{\alpha \beta}^p (\var)u_p(\var) + \Gamma_{\alpha \beta}^p (\var)\d_3u_p(\var), \\
		\Gamma_{\alpha \beta}^p (\var)\d_3u_p(\var) &\to 0 \en \WHMO,
		\end{align*}
		since $\d_3 u_p(\var)\to 0$ in $\WHMO$ (see step $(ii)$ and (\ref{ref8})). Now, the estimates  $||\d_3 \Gamma^p_{\alpha\beta}(\var)||_{0,\infty, \bar{\Omega}}\leq C\var$, with a constant $C>0$ (see Theorem \ref{Th_simbolos2D_3D}) and the convergences  $\var\bu(\var)\to 0$ in $H^{1}(0,T; [L^2(\Omega)]^3)$ (see step $(vi)$) imply that 
		\begin{align*}
		\d_3\Gamma_{\alpha \beta}^p (\var)u_p(\var) \to 0 \en \WLO.
		\end{align*}
		
		Therefore, $e_{\alpha\beta}(\d_3 \bu'(\var))=\d_3 e_{\alpha\beta}(\bu(\var))\to 0$ in $\WHMO$. So that, we complete the proof of the convergences (\ref{conv2}). Then, together with (\ref{conv1}) we have, by step $(ix)$, that  $\d_3 u_\alpha (\var) \to 0$ in $\WLO$. Since $\d_3 u_3(\var)\to 0$ in $\WLO$ and $\gab(\overline{\bu(\var)})\to \meab$ in $\WLo$ by steps $(ii)$ and $(vii)$, we can conclude from the identity (\ref{ref18}) the proof of this step.

		
		\item {\em All family $(\bu(\var))_{\var>0}$ converges strongly to $\bu$ in the space $\WVMO$.}
		
		The family $(\overline{\bu(\var)})_{\var>0}$ converges strongly  in $\WVMo$ by step $(viii)$ and $\d_3\bu(\var)\to \bcero$ in $\WLO$ for a subsequence (see steps $(ii)$ and $(x)$). Then, since the limit of such subsequence is unique, the whole family $(\d_3\bu(\var))_{\var>0}$ converges in $\WLO$. Therefore, by the definition of the norm $\left|\cdot\right| _{T,\Omega}^M$, $(\bu(\var))_{\var>0}$ is a Cauchy sequence in the Hilbert space $\WVMO$, hence, the conclusion follows.

	\end{enumerate}
	Therefore, the proof of the theorem is complete.
\end{proof}

\begin{remark}
	For each $\var>0$, let $\sigma^{ij,\var}=A^{ijkl,\var}\eij^\var(\bu^\var) + B^{ijkl,\var}\eij^\var(\dot{\bu}^\var)$ denote the contravariant components of the linearized stress tensor field for a family of linearly viscoelastic shells that satisfy the conditions of Theorem \ref{Th_convergencia} and let us define the scaled stresses $\sigma^{ij}(\var):\bar{\Omega}\to\mathbb{R}$ by letting $\sigma^{ij,\var}(\bx^\var)=:\sigma^{ij}(\var)(\bx)$ for all $\bx^\var=\pi^\var(\bx)\in\bar{\Omega}^\var$. Then, the scaled stresses satisfy  \begin{align*} 
	\sigma^{ij}(\var)=A^{ijkl}(\var)\eij(\var) + B^{ijkl}(\var)\deij(\var). 
	\end{align*}
	Hence, using the asymptotic behaviour of  $A^{ijkl}(\var)$, $B^{ijkl}(\var)$ (see Theorem \ref{Th_comportamiento asintotico}) and the strong convergences of $\eij(\var)(t,\cdot)$ in $\WLO$  found in Theorem \ref{Th_convergencia}, we can prove that  $\sigma^{ij}(\var)$  converge in $\LLO$.  To obtain these results we follow similar arguments to those used in \cite{Collard} for the elastic case. While for the elliptic membrane case (see \cite{eliptico}) we can prove that those convergences lead us to the plane stress case, the generalized membranes are subjected to   the admissible forces consideration.    
\end{remark}

It remains to prove the analogous result to the previous theorem in terms of de-scaled unknowns. Therefore we need to de-scale the unknown $\bxi$, solution of the two-dimensional variational scaled problem. By the scaling proposed in Section \ref{seccion_dominio_ind}, we define for each $\var>0$ the vector field $\bxi^\var$ such that
\begin{align*}
\bxi^\var:=\bxi \en \WVMo,
\end{align*}
that is solution os the de-scaled version of Problem \ref{problema_primertipo}:
\begin{problem}\label{problema_primertipo_esc}
	Find $\bxi^\var(t,\cdot)\in V_M^{\#}(\omega) \forallt $ such that,
	\begin{align*}
	&B_{M}^{\# \var}(\bxi^\var(t),\beeta)=L_{M}^{\#\var}(\beeta)(t) \ \forall \beeta\in V_M^{\#}(\omega), \ae,
	\\ \nonumber
	&\bxi^\var(0,\cdot)=\bxi_0^\var(\cdot),
	\end{align*}
	where $B_{M}^{\#\var}$ and $L_{M}^{\#\var}$ are the unique continuous extensions from 
	$H^{1}(0,T; V(\omega))$ to $\WVMo$ and from $V(\omega)$ to $V_M^{\#}(\omega)$ 
	of the functions  $B_{M}^\var:H^{1}(0,T; V(\omega)) \times V(\omega)\longrightarrow \mathbb{R}$ and $L_{M}^\var(t): V(\omega)\longrightarrow \mathbb{R}$, respectively, defined by
	\begin{align*} \nonumber
	B_{M}^\var(\bxi^\var(t),\beeta)&:=\var\int_{\omega}\aeps\gst(\bxi^\var(t))\gab(\beeta)\sqrt{a}dy + \var\int_{\omega}\beps\gst(\dot{\bxi^\var}(t))\gab(\beeta)\sqrt{a}dy 
	\\
	\qquad & - \var\int_0^te^{-k(t-s)}\into \ceps \gst(\bxi^\var(s))\gab(\beeta)\sqrt{a}dyds,
	\\
	L_{M}^\var(\beeta)(t)&:= \int_{\omega} \varphi^{\alpha\beta,\var}(t)\gab(\beeta)\sqrt{a}dy,  
	\end{align*}
	
	where  $\aeps,$ $\beps $ and $\ceps$ denote  the  re-scaled versions of the contravariant components of the two-dimensional fourth order tensors that we shall recall later (\ref{tensor_a_bidimensional})--(\ref{tensor_c_bidimensional}),  $\varphi^{\alpha\beta,\var}$ is a de-scaled version of the real function defined in (\ref{phi_1}).
	
\end{problem}

Notice that, for the viscoelastic generalized membrane shells, we can not consider the de-scaling of each component of the unknown separately, since the previous equality must be understood only in the abstract completion space. Therefore, we can prove the following convergence result:

\begin{theorem}
	Assume that $\btheta\in\mathcal{C}^3(\bar{\omega};\mathbb{R}^3)$. Consider a family of viscoelastic generalized membrane shells of the first kind with thickness $2\var$ approaching zero and with each having the same  middle surface $S=\btheta(\bar{\omega})$, with each subjected to a boundary condition of place along a portion of its lateral face having the same set $\btheta(\gamma_0)$ as its middle curve and subjected to admissible forces (see Section \ref{seccion_fuerzas_admisibles}).
	
	Let $\bu^\var=(u_i^\var)\in H^{1}(0,T; V(\Omega^\var))$ and $\bxi^\var=(\xi_i^\var)\in H^{1}(0,T; V_M^{\#}(\omega))$ respectively denote for each $\var>0$ the solutions to the three-dimensional and two-dimensional Problems \ref{problema_eps} and \ref{problema_primertipo_esc}. Moreover, let $\bxi=(\xi_i)\in H^{1}(0,T; V_M(\omega))$ denote the solution to the Problem \ref{problema_primertipo}. Then we have that
	\begin{align*}
	\bxi^\var=\bxi \ \textrm{and} \
	\frac{1}{2\var}\int^\var_{-\var} \bu^\var  dx_3^\var \to \bxi \ \textrm{in} \ \WVMo \ \textrm{as} \ \var \to 0,
	\end{align*}
\end{theorem}
\begin{proof}
	Notice that,
	\begin{align*}
	\frac{1}{2\var}\int^\var_{-\var} \bu^\var  dx_3^\var = \frac{1}{2}\int^1_{-1} \bu(\var)  dx_3= \overline{\bu(\var)}. 
	\end{align*}
	Hence, the conclusion follows by applying Theorem \ref{Th_convergencia}.
\end{proof}


\section{Conclusions} \label{conclusiones}

We have found and mathematically justified a model for  viscoelastic generalized membrane shells subjected to admissible forces. To this end we used the  asymptotic expansion method (presented in our previous work \cite{intro2}) and we have justified this approach by obtaining convergence theorems. As in the elastic case we have distinguished two cases (generalized membrane of the first kind or second kind) depending on whether or not the space  $V_0(\omega)$ contains non-zero functions. For each case, completion spaces were needed  in order to obtain well posed problems. 

The main novelty that this model presented is a long-term memory, represented by an integral on the time variable, more specifically
\begin{align*}
M(t,\beeta)=\int_0^te^{-k(t-s)}\into \c \gst(\bxi(s))\gab(\beeta)\sqrt{a}dyds , 
\end{align*}
for all $\beeta\in V_M^{\#}(\omega)$ (analogously for  generalized membranes of the second kind). Analogous behaviour has been detected in beam models for the bending-stretching of viscoelastic rods \cite{AV}, obtained by using asymptotic methods as well.  Also, this kind of viscoelasticity has been described in \cite{DL,Pipkin}, for example. 

As the viscoelastic case differs from the elastic case on time dependent constitutive law and external forces, we must consider the possibility that these models and the convergence result generalize the elastic case (studied in \cite{Ciarlet4b,CiarletLods5}). However, the reader can easily check that when the ordinary differential equation (\ref{ecuacion_casuistica1}) and (\ref{ecuacion_casuistica}) were presented, we  had to consider assumptions that make it impossible to include the elastic case.  Hence, the viscoelastic and elastic problems must be treated separately in order to reach reasonable and justified conclusions.

Furthermore,  as in the elastic case \cite{Ciarlet4b,CiarletLods5},  we found that $\d_3\bu(\var)(t, \cdot)\rightarrow \bcero \forallt$, while for the elliptic case we proved that the three-dimensional limit $\bu$ was independent of $x_3$ (see \cite{eliptico,eliptico2}). Therefore, the displacements on a viscoelastic generalized membrane shell might not be independent of the transversal variable. Moreover, notice that we proved convergence theorems when applied admissible forces (\ref{fuerzasadmisibles_def}) are considered, hence, the body and surface forces can not be arbitrarily chosen as in the elliptic case.

These models together with the elliptic case presented in our previous paper \cite{eliptico}, complete the study for the viscoelastic membrane shells. The remaining  case, when $V_F(\omega)$ contains non-zero functions (see \cite{intro2}), known as the problem of a viscoelastic flexural shell, has been studied in \cite{flexural}.

As future work,  we are currently working on the justification of the viscoelastic Koiter's  equations in \cite{Koiter}. Also, it would be interesting to derive error estimates for the two-dimensional models derived in \cite{intro2}. These results, we would prove the accuracy of our two-dimensional models for its application on real problems. 

\section*{Acknowledgements}
{\footnotesize \noindent This research was partially supported by Ministerio de Econom\'ia y Competitividad of Spain, under the grant  MTM2016-78718-P, with the participation of FEDER.}




%

\section*{References}
\bibliographystyle{abbrv}
\bibliography{biblio_tesis}

\begin{thebibliography}{10}
\expandafter\ifx\csname url\endcsname\relax
  \def\url#1{\texttt{#1}}\fi
\expandafter\ifx\csname urlprefix\endcsname\relax\def\urlprefix{URL }\fi
\expandafter\ifx\csname href\endcsname\relax
  \def\href#1#2{#2} \def\path#1{#1}\fi

\bibitem{Lions}
J.-L. Lions, Perturbations singuli\`eres dans les probl\`emes aux limites et en
  contr\^ole optimal, Lecture Notes in Mathematics, Vol. 323, Springer-Verlag,
  Berlin-New York, 1973.

\bibitem{CD}
P.~G. Ciarlet, P.~Destuynder, A justification of the two-dimensional linear
  plate model, J. M\'ecanique 18~(2) (1979) 315--344.

\bibitem{Destuynder}
P.~Destuynder, Sur une justification des mod\`eles de plaques et de coques par
  les m\'ethodes asymptotiques, Ph.D. thesis, Univ. P. et M. Curie, Paris
  (1980).

\bibitem{Ciarlet4b}
P.~G. Ciarlet, Mathematical elasticity. {V}ol. {III}: Theory of shells, Vol.~29
  of Studies in Mathematics and its Applications, North-Holland Publishing Co.,
  Amsterdam, 2000.

\bibitem{CiarletLods}
P.~G. Ciarlet, V.~Lods, On the ellipticity of linear membrane shell equations,
  J. Math. Pures Appl. 75 (1996) 107--124.

\bibitem{CiarletLods2}
P.~G. Ciarlet, V.~Lods, Asymptotic analysis of linearly elastic shells. {I}.
  {J}ustification of membrane shell equations, Arch. Rational Mech. Anal. 136
  (1996) 119--161.

\bibitem{CiarletLods5}
P.~G. Ciarlet, V.~Lods, Asymptotic analysis of linearly elastic shells:
  "generalized membrane shells", Journal of Elasticity 43 (1996) 147--188.

\bibitem{CiarletLods4}
P.~G. Ciarlet, V.~Lods, Asymptotic analysis of linearly elastic shells. {II}.
  {J}ustification of flexural shell equations, Arch. Rational Mech. Anal. 136
  (1996) 119--161.

\bibitem{Limin_mem}
X.~Li-ming, Asymptotic analysis of dynamic problems for linearly elastic shells
  - justification of equations for dynamic membrane shells, Asymptotic Analysis
  17 (1998) 121--134.

\bibitem{Limin_flex}
X.~Li-ming, Asymptotic analysis of dynamic problems for linearly elastic shells
  - justification of equations for dynamic flexural shells, Chinese Annals of
  Mathematics 22B (2001) 13--22.

\bibitem{Limin_koit}
X.~Li-ming, Asymptotic analysis of dynamic problems for linearly elastic shells
  - justification of equations for dynamic {K}oiter shells, Chinese Annals of
  Mathematics 22B (2001) 267--274.

\bibitem{ArosObs}
{\'A}.~Rodr{\'i}guez-Ar{\'o}s, Mathematical justification of an elastic
  elliptic membrane obstacle problem, Comptes Rendus Mecanique 345 (2017)
  153--157.
\newblock \href
  {http://dx.doi.org/http://dx.doi.org/10.1016/j.crme.2016.10.014}
  {\path{doi:http://dx.doi.org/10.1016/j.crme.2016.10.014}}.

\bibitem{ARA2018}
{\'A}.~Rodr{\'i}guez-Ar{\'o}s,
  \href{https://doi.org/10.1007/s10659-017-9638-1}{Models of elastic shells in
  contact with a rigid foundation: An asymptotic approach}, Journal of
  Elasticity 130~(2) (2018) 211--237.
\newblock \href {http://dx.doi.org/10.1007/s10659-017-9638-1}
  {\path{doi:10.1007/s10659-017-9638-1}}.
\newline\urlprefix\url{https://doi.org/10.1007/s10659-017-9638-1}

\bibitem{AA_contact_shell}
{\'A}.~Rodr\'{\i}guez-Ar\'os,
  \href{https://doi.org/10.1080/00036811.2017.1337894}{Mathematical
  justification of the obstacle problem for elastic elliptic membrane shells},
  Applicable Analysis 97~(8) (2018) 1261--1280.
\newblock \href
  {http://arxiv.org/abs/https://doi.org/10.1080/00036811.2017.1337894}
  {\path{arXiv:https://doi.org/10.1080/00036811.2017.1337894}}, \href
  {http://dx.doi.org/10.1080/00036811.2017.1337894}
  {\path{doi:10.1080/00036811.2017.1337894}}.
\newline\urlprefix\url{https://doi.org/10.1080/00036811.2017.1337894}

\bibitem{Cao_Aros}
M.~T. Cao-Rial, A.~Rodr\'{\i}guez-Ar\'{o}s,
  \href{https://doi.org/10.1016/j.nonrwa.2019.01.009}{Asymptotic analysis of
  unilateral contact problems for linearly elastic shells: error estimates in
  the membrane case}, Nonlinear Anal. Real World Appl. 48 (2019) 40--53.
\newblock \href {http://dx.doi.org/10.1016/j.nonrwa.2019.01.009}
  {\path{doi:10.1016/j.nonrwa.2019.01.009}}.
\newline\urlprefix\url{https://doi.org/10.1016/j.nonrwa.2019.01.009}

\bibitem{ArosCao19}
A.~Rodr\'{\i}guez-Ar\'{o}s, M.~T. Cao-Rial,
  \href{https://doi.org/10.1007/s00033-018-1008-8}{Asymptotic analysis of
  linearly elastic shells in normal compliance contact: convergence for the
  elliptic membrane case}, Z. Angew. Math. Phys. 69~(5) (2018) Art. 115, 22.
\newblock \href {http://dx.doi.org/10.1007/s00033-018-1008-8}
  {\path{doi:10.1007/s00033-018-1008-8}}.
\newline\urlprefix\url{https://doi.org/10.1007/s00033-018-1008-8}

\bibitem{TGAS_2020}
M.~T. Cao-Rial, G.~Casti{\~n}eira, A.~Rodr\'{\i}guez-Ar\'{o}s, S.~Roscani,
  Mathematical and asymptotic analysis of thermoelastic shells in normal damped
  response contact, (submitted).

\bibitem{DL}
G.~Duvaut, J.-L. Lions, Inequalities in Mechanics and Physics, Springer Berlin,
  1976.

\bibitem{LC1990}
J.~Lemaitre, J.~L. Chaboche, Mechanics of solid materials, Cambridge University
  Press, 1990.

\bibitem{Pipkin}
A.~C. Pipkin, Lectures in Viscoelasticity Theory, Applied Sciences,
  Springer-Verlag, New York, 1972.

\bibitem{contact2}
W.~Han, M.~Shillor, M.~Sofonea,
  \href{http://www.sciencedirect.com/science/article/pii/S037704270000707X}{Variational
  and numerical analysis of a quasistatic viscoelastic problem with normal
  compliance, friction and damage}, Journal of Computational and Applied
  Mathematics 137~(2) (2001) 377 -- 398.
\newblock \href
  {http://dx.doi.org/http://dx.doi.org/10.1016/S0377-0427(00)00707-X}
  {\path{doi:http://dx.doi.org/10.1016/S0377-0427(00)00707-X}}.
\newline\urlprefix\url{http://www.sciencedirect.com/science/article/pii/S037704270000707X}

\bibitem{contact3}
W.~Han, M.~Sofonea,
  \href{http://www.sciencedirect.com/science/article/pii/S0377042700006270}{Time-dependent
  variational inequalities for viscoelastic contact problems}, Journal of
  Computational and Applied Mathematics 136 (2001) 369 -- 387.
\newblock \href
  {http://dx.doi.org/http://dx.doi.org/10.1016/S0377-0427(00)00627-0}
  {\path{doi:http://dx.doi.org/10.1016/S0377-0427(00)00627-0}}.
\newline\urlprefix\url{http://www.sciencedirect.com/science/article/pii/S0377042700006270}

\bibitem{SofoneaAMS}
W.~Han, M.~Sofonea, Quasistatic Contact Problems in Viscoelasticity and
  Viscoplascticity, Studies in Advanced Mathematics, American Mathematical
  Society- Intl., Press, 2002.

\bibitem{contact7}
J.~Jaru{\v{s}}ek, Contact problems with given time-dependent friction force in
  linear viscoelasticity, Commentationes Mathematicae Universitatis Carolinae
  31 (2) (1990) 257--262.

\bibitem{contact1}
S.~Mig\'orski, A.~Ochal,
  \href{http://www.sciencedirect.com/science/article/pii/S0362546X04005619}{Hemivariational
  inequality for viscoelastic contact problem with slip-dependent friction},
  Nonlinear Analysis: Theory, Methods \& Applications 61 (2005) 135 -- 161.
\newblock \href {http://dx.doi.org/http://dx.doi.org/10.1016/j.na.2004.11.018}
  {\path{doi:http://dx.doi.org/10.1016/j.na.2004.11.018}}.
\newline\urlprefix\url{http://www.sciencedirect.com/science/article/pii/S0362546X04005619}

\bibitem{ASV06}
{\'A}.~Rodr{\'{\i}}guez-Ar{\'o}s, M.~Sofonea, J.~M. Via{\~n}o, Numerical
  analysis of a frictional contact problem for viscoelastic materials with
  long-term memory, Numerical mathematics and advanced applications (2006)
  1099--1107.

\bibitem{jaru1}
I.~Bock, J.~Jaru{\v{s}}ek, Unilateral dynamic contact of viscoelastic von
  {K}\'arm\'an plates, Adv. Math. Sci. Appl. 16~(1) (2006) 175--187.

\bibitem{plate12}
I.~Bock, J.~Lov\'i{\v{s}}ek,
  \href{http://www.sciencedirect.com/science/article/pii/S0378475402000952}{On
  unilaterally supported viscoelastic von {K}\'arm\'an plates with a long
  memory}, Mathematics and Computers in Simulation 61 (2003) 399--407.
\newblock \href
  {http://dx.doi.org/http://dx.doi.org/10.1016/S0378-4754(02)00095-2}
  {\path{doi:http://dx.doi.org/10.1016/S0378-4754(02)00095-2}}.
\newline\urlprefix\url{http://www.sciencedirect.com/science/article/pii/S0378475402000952}

\bibitem{plate15}
I.~Bock,
  \href{http://www.sciencedirect.com/science/article/pii/0377042795000828}{On
  von {K}\'arm\'an equations for viscoelastic plates}, Journal of Computational
  and Applied Mathematics 63~(1) (1995) 277 -- 282.
\newblock \href
  {http://dx.doi.org/http://dx.doi.org/10.1016/0377-0427(95)00082-8}
  {\path{doi:http://dx.doi.org/10.1016/0377-0427(95)00082-8}}.
\newline\urlprefix\url{http://www.sciencedirect.com/science/article/pii/0377042795000828}

\bibitem{plate13}
J.~E. Mu{\~{n}}oz-Rivera, G.~Perla-Menzala, Decay rates of solutions to a von
  {K}\'arm\'an system for viscoelastic plates with memory, Quart. Appl. Math.
  57 (1) (1999) 181--200.

\bibitem{plate14}
I.~Bock,
  \href{http://www.sciencedirect.com/science/article/pii/S037847549900066X}{On
  large deflections of viscoelastic plates}, Mathematics and Computers in
  Simulation 50 (1999) 135--143.
\newblock \href
  {http://dx.doi.org/http://dx.doi.org/10.1016/S0378-4754(99)00066-X}
  {\path{doi:http://dx.doi.org/10.1016/S0378-4754(99)00066-X}}.
\newline\urlprefix\url{http://www.sciencedirect.com/science/article/pii/S037847549900066X}

\bibitem{plate16}
Y.~A. Rossikhin, M.~Shitikova,
  \href{http://www.sciencedirect.com/science/article/pii/S0020746205001034}{Analysis
  of free non-linear vibrations of a viscoelastic plate under the conditions of
  different internal resonances}, International Journal of Non-Linear Mechanics
  41~(2) (2006) 313--325.
\newblock \href
  {http://dx.doi.org/http://dx.doi.org/10.1016/j.ijnonlinmec.2005.08.002}
  {\path{doi:http://dx.doi.org/10.1016/j.ijnonlinmec.2005.08.002}}.
\newline\urlprefix\url{http://www.sciencedirect.com/science/article/pii/S0020746205001034}

\bibitem{shell11}
I.~Bock, J.~Jaru{\v{s}}ek, System modeling and optimization, Springer,
  Heidelberg, 2013.

\bibitem{Quarteroni}
L.~Formaggia, D.~Lamponi, A.~Quarteroni, One-dimensional models for blood flow
  in arteries, Journal of Engineering Mathematics 47 (2003) 251--276.

\bibitem{intro}
G.~Casti{\~n}eira, {\'A}.~Rodr{\'i}guez-Ar{\'o}s,
  \href{https://export.arxiv.org/abs/1604.02280v2}{Derivation of models for
  linear viscoelastic shells by using asymptotic analysis}, Preprint.
\newline\urlprefix\url{https://export.arxiv.org/abs/1604.02280v2}

\bibitem{AV}
{\'A}.~Rodr{\'{\i}}guez-Ar{\'o}s, J.~M. Via{\~n}o,
  \href{http://dx.doi.org/10.1016/j.jmaa.2010.04.067}{Mathematical
  justification of viscoelastic beam models by asymptotic methods}, J. Math.
  Anal. Appl. 370~(2) (2010) 607--634.
\newblock \href {http://dx.doi.org/10.1016/j.jmaa.2010.04.067}
  {\path{doi:10.1016/j.jmaa.2010.04.067}}.
\newline\urlprefix\url{http://dx.doi.org/10.1016/j.jmaa.2010.04.067}

\bibitem{AV2}
{\'A}.~Rodr{\'{\i}}guez-Ar{\'o}s, J.~M. Via{\~n}o,
  \href{http://www.springerlink.com/content/dr63883430665033/}{Mathematical
  justification of {K}elvin-{V}oigt beam models by asymptotic methods}, Z.
  Angew. Math. Phys 63~(3) (2012) 529--556.
\newblock \href {http://dx.doi.org/10.1007/s00033-011-0180-x}
  {\path{doi:10.1007/s00033-011-0180-x}}.
\newline\urlprefix\url{http://www.springerlink.com/content/dr63883430665033/}

\bibitem{eliptico}
G.~Casti{\~n}eira, {\'A}.~Rodr{\'i}guez-Ar{\'o}s, On the justification of the
  viscoelastic elliptic membrane shell equations, Journal of Elasticity 130
  (2018) 85--113.
\newblock \href {http://dx.doi.org/10.1007/s00033-011-0180-x}
  {\path{doi:10.1007/s00033-011-0180-x}}.

\bibitem{eliptico2}
G.~Casti{\~n}eira, {\'A}.~Rodr{\'i}guez-Ar{\'o}s, Mathematical justification of
  a viscoelastic elliptic membrane problem, Comptes Rendus Mecanique 345 (2017)
  824--831.
\newblock \href {http://dx.doi.org/10.1007/s00033-011-0180-x}
  {\path{doi:10.1007/s00033-011-0180-x}}.

\bibitem{Shillor}
M.~Shillor, M.~Sofonea, J.~Telega, Models and Analysis of Quasistatic Contact,
  Lecture Notes in Physics, Vol. 655, Springer Berlin, 2004.

\bibitem{Collard}
C.~Collard, B.~Miara, Asymptotic analysis of the stresses in thin elastic
  shells, Arch. Rational Mech. Anal. 148 (1999) 233--264.

\bibitem{flexural}
G.~Casti{\~n}eira, {\'A}.~Rodr{\'i}guez-Ar{\'o}s, On the justification of the
  viscoelastic flexural shell equations, Computers and Mathematics with
  Applications 77(11) (2019) 2933--2942.

\bibitem{Koiter}
G.~Casti{\~n}eira, {\'A}.~Rodr{\'i}guez-Ar{\'o}s, On the justification of
  {K}oiter's equations for viscoelastic shells, Preprint.

\bibitem{intro2}
G.~Casti{\~n}eira, {\'A}.~Rodr{\'i}guez-Ar{\'o}s,
  \href{https://export.arxiv.org/abs/1604.02280v2}{Linear viscoelastic shells:
  an asymptotic approach}, Asymptotic Analysis 107 (2018) 169--201.
\newline\urlprefix\url{https://export.arxiv.org/abs/1604.02280v2}

\end{thebibliography}

\end{document}